\documentclass[a4paper,twoside,10pt]{article}

\newcommand{\mytitle}{Quantitative OCT reconstructions for dispersive media}
\title{\mytitle}

\date{November 26, 2019}
\author{Peter Elbau$^1$\\{\footnotesize\href{mailto:peter.elbau@univie.ac.at}{peter.elbau@univie.ac.at}}
\and Leonidas Mindrinos$^2$\\{\footnotesize\href{mailto:leonidas.mindrinos@ricam.oeaw.ac.at}{leonidas.mindrinos@ricam.oeaw.ac.at}}
\and Leopold Veselka$^1$\\{\footnotesize\href{mailto:lepold.veselka@univie.ac.at}{leopold.veselka@univie.ac.at}}}

\usepackage[utf8]{inputenc}

\usepackage{microtype}

\usepackage{amsmath}

\usepackage{longtable}

\usepackage{amssymb}

\usepackage{mathrsfs}

\usepackage{mathtools}

\usepackage{siunitx}

\usepackage[a4paper,centering,bindingoffset=0cm,marginpar=0cm]{geometry}

\usepackage{titlesec}
\usepackage[font=footnotesize,format=plain,labelfont=sc,textfont=sl,width=0.75\textwidth,labelsep=period]{caption}
\titleformat{\section}{\filcenter\sc\large}{\thesection.\;}{0em}{}
\titleformat{\subsection}[runin]{\bf}{\thesubsection.\;}{0em}{}[.]

\newpagestyle{headers}%
{%
	\headrule
	\sethead%
		[\thepage]%
		[\footnotesize\sc P.~Elbau, L.~Mindrinos, and L.~Veselka]%
		[]%
		{}%
		{\footnotesize\sc{Quantitative OCT reconstructions for dispersive media}}%
		{\thepage}
}
\usepackage[backend=bibtex,maxnames=15]{biblatex}
\usepackage[pdftex,colorlinks=true,linkcolor=blue,citecolor=green,urlcolor=blue,bookmarks=true,bookmarksnumbered=true]{hyperref}
\hypersetup
{
    pdfauthor={Peter Elbau, Leonidas Mindrinos, Leopold Veselka},
    pdfsubject={},
    pdftitle={\mytitle},
    pdfkeywords={}
}

\usepackage[hyperref,amsmath,thmmarks]{ntheorem}
\usepackage{aliascnt}
\theoremstyle{break}
\theoremseparator{.}
\theoremheaderfont{\kern-1em\normalfont\bfseries}
\newtheorem{lemma}{Lemma}[section]

\newaliascnt{proposition}{lemma}
\newtheorem{proposition}[proposition]{Proposition}
\aliascntresetthe{proposition}

\newaliascnt{corollary}{lemma}

\aliascntresetthe{corollary}

\newaliascnt{assumptions}{lemma}

\aliascntresetthe{assumptions}

\newaliascnt{invpro}{lemma}

\aliascntresetthe{invpro}

\theorembodyfont{\normalfont}
\newaliascnt{definition}{lemma}

\aliascntresetthe{definition}

\theorembodyfont{\normalfont}
\newaliascnt{example}{lemma}
\newtheorem{example}[example]{Example}
\aliascntresetthe{example}

\theorembodyfont{\normalfont}
\newaliascnt{convention}{lemma}

\aliascntresetthe{convention}

\theorembodyfont{\normalfont\itshape}
\theoremseparator{:}
\theoremheaderfont{\kern-1em\normalfont\itshape}
\newaliascnt{remark}{lemma}
\newtheorem{remark}[remark]{Remark}
\aliascntresetthe{remark}

\theoremstyle{nonumberplain}
\theorembodyfont{\normalfont}
\newtheorem{proof}{Proof}

    \usepackage[ntheorem]{empheq}

\usepackage{bbm}
\usepackage{bm}
\newcommand{\R}{\mathbbm{R}}

\newcommand{\C}{\mathbbm{C}}
\newcommand{\N}{\mathbbm{N}}
\renewcommand{\b}{\bm}

\let\RE\Re
\let\Re=\undefined
\DeclareMathOperator{\Re}{\RE e}
\let\IM\Im
\let\Im=\undefined
\DeclareMathOperator{\Im}{\IM m}
\DeclareMathOperator{\sign}{sign}

\DeclareMathOperator{\supp}{supp}
\newcommand{\norm}[1]{\left\|#1\right\|}

\usepackage{dsfont}
\renewcommand{\b}{\mathds{1}}

\usepackage{subfigure}
\usepackage{caption}

\usepackage[ ]{algorithm2e}

\SetCommentSty{mycommfont}

\begin{document}

\maketitle
\hspace*{4em}
\parbox[t]{0.4\textwidth}{\footnotesize
\hspace*{-1ex}$^1$Faculty of Mathematics\\
University of Vienna\\
Oskar-Morgenstern-Platz 1\\
A-1090 Vienna, Austria}
\parbox[t]{0.4\textwidth}{\footnotesize
\hspace*{-1ex}$^2$Johann Radon Institute for Computational\\
\hspace*{0.1em}and Applied Mathematics (RICAM)\\
Altenbergerstrasse 69\\
A-4040 Linz, Austria}

\vspace*{2em}

\begin{abstract}
We consider the problem of reconstructing the position and the time-dependent optical properties of a linear dispersive medium from OCT measurements. The medium is multi-layered described by a piece-wise inhomogeneous refractive index. The measurement data are from a frequency-domain OCT system and we address also the phase retrieval problem. The parameter identification problem can be formulated as an one-dimensional inverse problem. Initially, we deal with a non-dispersive medium and we derive an iterative scheme that is the core of the algorithm for the frequency-dependent parameter. The case of absorbing medium is also addressed. 
\end{abstract}

\section{Introduction}\label{sec1}

Optical Coherence Tomography (OCT) is nowadays considered as a well-established imaging modality producing high-resolution images of biological tissues. Since it first appeared in the beginning of 1990s \cite{FerHitDreKamSat93, HuaSwaLinSchuStiCha91, SwaIzaHeeHuaLinSchu93}, OCT has gained increasing acceptance because of its non-invasive nature and the use of non-harmful radiation. Main applications remain tissue diagnostics and ophthalmology. It operates at the visible and near-infrared spectrum and the measurements consist mainly of the backscattered light from the sample. OCT is analogous to Ultrasound Tomography where acoustic waves are used and differs from Computed Tomography (where electromagnetic waves are also used) because of its limited penetration depth (few millimeters) due to the lower energy radiation. As OCT data we consider the measured intensity of the backscattered light at some detector area usually far from the medium. 

However, the intensity of light, undergoing few scattering events, is not measured directly, but the OCT setup is based on low coherence interferometry. The incoming broadband and continuous wave light passes through a beam-splitter and it is split into two identical beams. One part travels in a reference path and is totally back-reflected by a mirror and the second part is incident on the sample. The backscattered from the sample and the back-reflected light are recombined and their superposition is then measured at a detector. The maximum observed intensity refers to constructive interference, and this happens when the two beams travel equal lengths. For a detailed explanation of the experimental setup we refer to \cite{Fer10, TomWan05} and to the book \cite{DreFuj15}.

The way the measurements are performed characterizes and differentiates an OCT system. We summarize here the different setups considered in this work:

\begin{description}
\item[Time-domain OCT:] The reference mirror is moving and for each position a measurement is performed. By scanning the reference arm, different depth information from the sample is obtained. 

\item[Frequency-domain OCT:] The mirror is placed at a fixed position and the detector is replaced by a spectrometer, which captures the whole spectrum of the interference pattern. 

\item[State-of-the-art OCT:] The incoming light is focused, through objective lenses, to a specific region at a certain depth in the sample. The backscattered light is measured at a point detector. 
 
\item[Standard OCT:] The vector nature of light is ignored and the electromagnetic wave is treated as a scalar quantity. Then, only the total intensity is measured.
\end{description}

Time- and Frequency-domain OCT provide almost equivalent measurements that are connected through a Fourier transform. The advantage of the later is that no mechanical movement of the mirror is required, improving the acquisition time. The last two cases simplify the following mathematical analysis since we can consider scalar quantities and depth-dependent optical parameters. For an overview of the different mathematical models that can be used in OCT we refer to the book chapter \cite{ElbMinSch15}.

We consider Maxwell's equations to model the light propagation in the sample, which is assumed to be a linear isotropic dielectric medium. We deal with dispersive and non-dispersive media. Firstly, using a general representation for the initial illumination, we  present the direct problem of computing the OCT data, given the optical properties of the sample. Then, we derive reconstruction methods for solving the inverse problem of recovering the refractive index, real or complex valued. Motivated by the layer stripping algorithms \cite{Som94, SylWinGyl96}, we present a layer-by-layer reconstruction method that alternates between time and frequency domain and holds for dispersive media. 
 
Without loss of generality, the OCT system can be simplified  by placing the beam-splitter 
and the detector at the same position.  The medium 
is contained in a bounded domain $\Omega \subset \R^3,$ such that 
$\mbox{supp} \chi (t,\cdot) \subset \Omega,$ for all $t\in\R,$ where $\chi$ is the electric susceptibility,  a scalar quantity describing the optical properties of a linear dielectric medium. We set $\chi$ to zero for negative times. Also, the medium for $t\leq 0$ is assumed to be in a stationary state with zero stationary fields. Then, the electric field $E \in C^\infty (\R\times \R^3;\R^3)$ and the magnetic field $H \in C^\infty (\R\times \R^3;\R^3),$ in the absence of charges and currents, satisfy the Maxwell's equations
\begin{equation}\label{eq_max}
\nabla\times E (t,\mathrm{x}) + \frac1c \frac{\partial H}{\partial t } (t,\mathrm{x}) = 0, \quad  \nabla\times H (t,\mathrm{x}) - \frac1c \frac{\partial D}{\partial t } (t,\mathrm{x}) = 0,  
\end{equation}
where $c$ is the speed of light and $D$ is the electric displacement, given by
\begin{equation}\label{eq_const}
D(t,\mathrm{x}) = E(t,\mathrm{x}) + 4\pi \int_\R \chi (\tau, \mathrm{x}) E (t-\tau, \mathrm{x}) d \tau.
\end{equation}
This relation models a linear dielectric, dispersive medium with inhomogeneous, isotropic and non-stationary parameter.

The two identical laser pulses, one incident on the sample and the other on the mirror, are described initially, before the time $t=0,$ as vacuum solutions of the Maxwell's equations, meaning \eqref{eq_max} for $D\equiv E,$ defined by
$E_0, \, H_0 \in C^\infty (\R\times \R^3;\R^3).$ In practice, the medium is illuminated by a Gaussian light, however at the scale of the sample the laser pulse can be approximated by a linearly polarized plane wave  \cite{Fer96}. We assume that the incident wave does not interact with the medium until $t=0,$ resulting to the condition
\begin{equation}\label{eq_initial}
E (t,\mathrm{x}) = E_0 (t,\mathrm{x}), \quad  H (t,\mathrm{x}) = H_0 (t,\mathrm{x}), \quad t <0, \, \mathrm{x} \in \R^3.
\end{equation}

The mirror is modeled as a medium with (infinitely) large constant electric susceptibility, with surface given by the hyperplane placed at $r\in\R$ distance from the source. Given the form of the incident wave, the reference field (back-reflected field), denoted by $E_r,$ can be explicitly calculated. 

The sample wave (backscattered wave) is given as a solution of the system 
\eqref{eq_max} -- \eqref{eq_initial}.  Then, the two backward traveling waves are recombined at the beam splitter, assumed to be at the detector position. In time-domain OCT, the sum of these two fields, integrated over all times, is measured at each point of the two-dimensional detector array $\mathcal D \subset \R^2$.   Thus, as observed quantity we consider
\begin{equation}\label{eqEffectiveMeasurements}
	\int_\R  |(E-E_0)(t,\mathrm{x}) + (E_r -E_0) (t,\mathrm{x}) |^2 dt, \quad r \in \R, \, \mathrm{x}\in\mathcal D.
\end{equation}

Under some assumptions on the incident field \cite{ElbMinSch15}, we may recover from the above measurements, the quantity
\begin{equation}\label{data_oct_time}
(\hat E-\hat E_0)(\omega ,\mathrm{x}), \quad \omega \in \R, \, \mathrm{x}\in\mathcal D,
\end{equation}
where $\hat f = \mathcal{F} (f)$ denotes the Fourier transform of $f$ with respect to time
\begin{equation*}
\mathcal{F} (f) (\omega) = \int_\R f(t) e^{i \omega t } dt.
\end{equation*}

In frequency-domain OCT, the detecting scheme is different. The mirror is not moving  ($r$ is fixed) and the detector is replaced by a spectrometer. Then, the intensity of the sum of the Fourier transformed fields at every available frequency (corresponding to different pixels at the CCD camera) is measured
\begin{equation}\label{eq_fourier_data}
\hat m (\omega, \mathrm{x}) =  |(\hat E-\hat E_0)(\omega ,\mathrm{x}) + (\hat E_r -\hat E_0) (\omega ,\mathrm{x}) |^2, \quad \omega \in \R, \, \mathrm{x}\in\mathcal D.
\end{equation}

In practice, we obtain data only for few frequencies restricted by the limited bandwidth of the spectrometer. The OCT system allows also for measurements of the intensities of the two fields independently, by blocking one arm at a time. Thus, we assume that the quantity 
\begin{equation}\label{eq_fourier_data2}
\hat m_s(\omega, \mathrm{x}) =   |(\hat E-\hat E_0)(\omega, \mathrm{x}) |^2, \quad \omega \in \R, \, \mathrm{x}\in\mathcal D,
\end{equation}
is also available.  The main difference between the two setups is that \eqref{data_oct_time} provide us with the full information of the backscattered field, amplitude and phase, which is not the case in \eqref{eq_fourier_data2}, where we get phase-less data. We address later the problem of phase retrieval, meaning how to obtain \eqref{data_oct_time} from  \eqref{eq_fourier_data2}.

Up to now, what we have modeled is known as full-field OCT where the whole sample is illuminated by an extended field. The main problem is that we want to reconstruct a $(1+3)$-dimensional function $\chi$ from OCT data, either \eqref{eqEffectiveMeasurements} or \eqref{eq_fourier_data}, which are $(1+2)$-dimensional. Thus,  we have to impose additional assumptions in order to compensate for the lack of dimension. To solve this problem, we consider a medium which admits  a multi-layer structure. This assumption is not far from reality since OCT is mainly used in  ophthalmology (imaging the retina) and human skin imaging. In both cases the imaging object consists of multiple layers with varying properties and thicknesses \cite{HeeIzaSwaHua95, KriBar04}.

If the medium is non-dispersive, meaning that the optical parameter is stationary, the function $\chi$ can be modeled as a $\delta-$distribution in time, so that its Fourier transform (temporal) does not depend on frequency.  Then, even if we have enough information (theoretically), in OCT, as in any tomographic imaging technique, we deal with the problem of inverting partial and limited-angle data. This is the result of measuring only the back-scattered light for a limited frequency spectrum. In OCT, a narrow beam is used, resulting to an  almost monochromatic illumination centered around a frequency. 

In the following, we focus on data provided from a state-of-the-art and standard OCT system, where point-like illumination is used. In this case, only a small region inside the object is illuminated so that the function $\chi$ can be assumed depth-dependent and constant in the other two directions.  Again we assume that locally the illumination is still properly described by a plane wave. 

Let $\mathrm{x} = (x,y,z),$ where the $z-$direction denotes the depth direction. We model the light as a transverse electric polarized electromagnetic wave of the form
\begin{equation*}
E (t, \mathrm{x}) = \begin{pmatrix}
0 \\ u(t,z) \\ 0
\end{pmatrix}, \quad H (t, \mathrm{x}) = \begin{pmatrix}
v (t,z) \\  0 \\ w (t,z)
\end{pmatrix}.
\end{equation*} 
Then, the Maxwell's equations \eqref{eq_max} together with \eqref{eq_const} are simplified to 
\begin{equation}\label{eq_helm}
\Delta u (t,z) - \frac{1}{c^2} \frac{\partial^2}{\partial t^2} \int_\R \epsilon (\tau, z) u (t-\tau, z) d \tau = 0,
\end{equation}
for the scalar valued function $u,$ where $\Delta = \partial^2 / \partial z^2.$ Here, we define the time-dependent electric permittivity $\epsilon (t,z) = \delta (t) + 4\pi \chi (t,z),$ which varies also with respect to depth. The condition  \eqref{eq_initial} is replaced by
\begin{equation}\label{eq_initial2}
u (t,z) = u_0 (t,z), \quad   t <0, \, z \in \R.
\end{equation}

The medium admits a multi-layered structure with $N$ layers orthogonal to the $z-$direction, having spatial-independent but time-dependent refractive indices $n = \sqrt{\epsilon},$ and varying lengths. We define $L = \cup_{j=1}^N L_j$ and we set
\begin{equation}
\label{rei}
n (t, z) = \begin{cases}
n_0,   & z \in \R \setminus \overline{L},\\
n_j(t), & z \in L_j .
\end{cases}
\end{equation}

\begin{figure}[t]
\begin{center}
\includegraphics[width=0.9\textwidth]{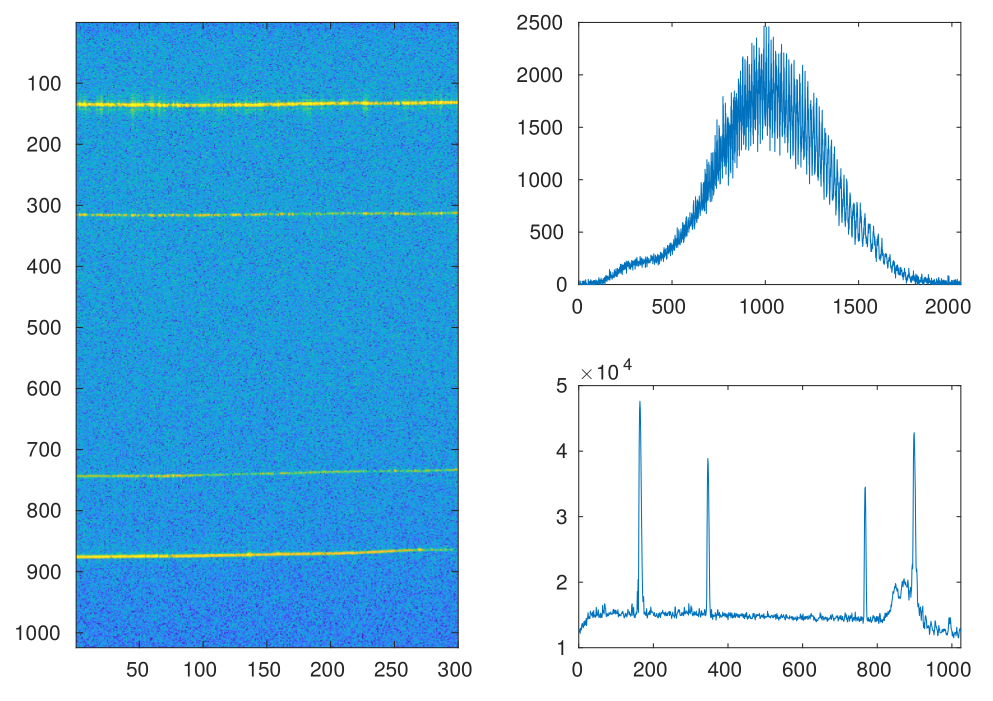}
\caption{Experimental data obtained from a frequency-domain OCT system of a three-layer medium with piece-wise constant refractive index.  Courtesy of Ryan Sentosa and Lisa Krainz, Medical University of Vienna.}\label{fig_real}
\end{center}
\end{figure}

This setup is commonly used for modeling the problem of parameter identification from OCT  data. Even the volumetric OCT data consist of multiple A-scans, which is a one-dimensional cross-sectional of the medium across the $z-$direction. Under the assumption of a layered medium, the multiple A-scans are averaged over the $x-$ and $y-$directions producing a profile of the measured intensity with respect to frequency or depth (post-processed image).  

In Fig. \ref{fig_real}, we see the experimental data for a three-layer medium with total length $0.7$mm, having  two layers (top and bottom) filled with Noa61 $(n_1 = n_3 \approx 1.55)$ and a middle one filled with DragonSkin $(n_2 \approx 1.405)$.  The spectrometer uses a grating with central wavelength $840$nm, going from $700$ to $960$nm. On top right, we see an A-scan of the $``$raw$"$ data (depth information), meaning the intensity of the combined sample and reference fields at a given point on the surface plane. The left picture is the post-processed B-scan (two-dimensional cross-sectional of the volumetric data) where we see clearly the four reflections from the boundaries. The bottom right picture presents  the averaged (over lateral dimension) post-processed version of the data on the left. All $x-$axes are in pixel units.

We refer to \cite{BruCha05, SeeMul14, ThrYurAnd00, TomWan06} for recent works using similar setup and assumptions. Our work differs from previous methods in that we consider dispersive medium. We deal also with absorbing medium, a property that is usually neglected. We address three different cases for the layered medium:
\begin{itemize}
\item $n_j (t) \equiv n_j, \quad j = 1,...,N$ (non-dispersive),
\item $ \hat n_j (\omega) \in \R, \quad j = 1,...,N$ (dispersive),
\item $ \hat n_j (\omega) \in \C, \quad j = 1,...,N$ (dispersive with absorption).
\end{itemize}

The paper is organized as follows. In Sec. \ref{sec2} we present the forward problem, meaning given the medium (location and properties) find the measurement data. We derive formulas that are also needed for the corresponding inverse problem, which we address in Sec. \ref{sec3}. Iterative schemes are presented for dispersive media and a mathematical model is given for the case of absorbing media. In Sec. \ref{sec4} we give numerical results for simulated data, and we show that the parameter identification problem can be solved under few assumptions.

\section{The forward problem}\label{sec2}

We derive mathematical models for the direct problem in OCT for multi-layer media with piece-wise inhomogeneous refractive index. We start with a single-layer medium and then we generalize to more layers. The multiple reflections are also taken into account. Most of the formulas presented in this section, like the solutions of the initial value problems or the reflection and transmission operators (analogue to the Fresnel equations) can be found in classic books on partial differential equations \cite{Eva10, Str07} and optics \cite{BorWol99, Che90, Hec02}, respectively. However, we summarize them here, on one hand because we want to derive a rigorous mathematical model in both time and frequency domains and on the other hand because they are needed for the corresponding inverse problems. 
The easier but essential time-independent case is treated first. Then, we consider the time-dependent case by moving to the frequency-domain for real and complex valued parameters. 

\subsection{Non-dispersive medium}\label{subsec_non}

Here, we simplify \eqref{rei}, and we consider the following form for the refractive index
\begin{equation}\label{eq_refract}
n (z) = \begin{cases}
n_0,   & z \in \R \setminus \overline{L},\\
n_j, & z \in L_j,
\end{cases}
\end{equation}
for $j=1,...,N.$ We describe the light propagation using \eqref{eq_helm} together with \eqref{eq_initial2}. Under the above assumption, we obtain 
\begin{equation}
\label{l53}
\partial_{tt}u(t,z) = \tfrac{c^2}{n^2(z)}\Delta u(t,z), \quad t\in\R,\,z\in\R,
\end{equation} 
the one-dimensional wave equation. In the following, we  use $c_j = c/n_j, \, j = 0,...,N.$ Let us assume that the initial field is given by the form 
\begin{equation}
\label{l1}
u_0(t,z) = f_0(z-  c_0 t),
\end{equation}
together with the assumption that $\mbox{supp} f_0 \subset (-\infty,z_1)$, where $z_1$ represents the surface (first boundary point) of the medium $L$. This assumption on the support of the function reflects the condition that the laser beam does not interact with the probe until time $t=0$. 

We model the single-layer medium as $L = (z_1, \,z_2),$ for $z_1 < z_2,$ but initially we consider the case
\begin{equation}\label{single_ref}
n(z)=\begin{cases}
n_0,& z<z_1, \\
n_1,& z>z_1.
\end{cases}
\end{equation} 

Then, we obtain the system  
\begin{equation}\label{l2}
\begin{aligned}
\partial_{tt}u(t,z) &= \tfrac{c^2}{n^2(z)}\Delta u(t,z), &t\in\R,\;z\in\R, \\
u(t,z) &= f_0(z-  c_0 t), &t<0,\;z\in\R. 
\end{aligned}
\end{equation}

\begin{figure}[t]
    \centering
    \begin{subfigure}
        \centering
        \includegraphics[scale=1]{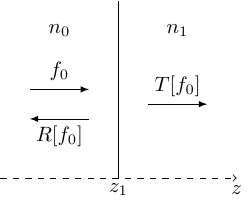}
    \end{subfigure}%
    \hspace{1.5cm}
    \begin{subfigure}
        \centering
        \includegraphics[scale=1]{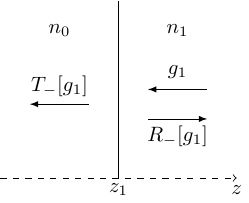}
    \end{subfigure}
    \caption{Wave propagation. The reflection and transmission operators for the sub-problem \eqref{l2} (left) and the sub-problem \eqref{l10} (right).
        }\label{fig_motion}
\end{figure}

The above system of equations describes a wave traveling from the left incident on the interface at $z_1 \in \R$, see the left picture of Fig. \ref{fig_motion}. It is easy to derive the solution, which is given by
\begin{equation}
\label{l4}
u(t,z)=\begin{cases}
u_{-}(t,z)=f_0(z- c_0 t)+g_0(z+ c_0 t), &t\geq 0, \, z<z_1,\\[1.5ex]
u_+ (t,z)=f_1(z- c_1 t), &t\geq 0, \,  z>z_1. 
\end{cases}
\end{equation}
Here, the function $g_0$ and $f_1$ describes the reflected and transmitted field, respectively. Given the continuity condition at $z=z_1$, meaning
\begin{equation*}
\lim_{z\uparrow z_1} u_-(t,z)=\lim_{z\downarrow z_1} u_+(t,z), \quad \lim_{z\uparrow z_1} \partial_z u_-(t,z)=\lim_{z\downarrow z_1} \partial_z u_+(t,z), 
\end{equation*}
we find a representation of $g_0$ and $f_1$ via operators. We denote the reflection operator by $R$ and the transmission operator by $T$, defined by
\begin{equation}\label{l7}
R: f_0 \mapsto g_0, \quad R[f_0](z+ c_0 t)=\frac{c_1-c_0}{c_1+c_0}f_0 \left(2z_1-(z+c_0 t)\right),
\end{equation}
and 
\begin{equation*}
T: f_0 \mapsto f_1, \quad T[f_0](z- c_1 t)=\frac{2c_1}{c_1+c_0}f_0\left(z_1+\frac{c_0}{c_1}((z-c_1 t)-z_1)\right).
\end{equation*}

The fact that $\mbox{supp} f_0\subset (-\infty,z_1)$ implies that for every $t<0$, $R[f_0]=0,$ in $(-\infty,z_1),$ and $T[f_0]=0,$ in $(z_1,\infty)$. This is true, since neither the reflected nor the transmitted wave exists before the interaction of the initial wave with the boundary. Finally, we define the operator $U_1 : f_0\mapsto u,$ mapping the initial function $f_0$ to the solution $u,$ given by \eqref{l4}, of the  problem \eqref{l2}.

Now we consider the following problem
\begin{equation}\label{l10}
\begin{aligned}
\partial_{tt}u(t,z) &= \tfrac{c^2}{n^2(z)}\Delta u(t,z), &t\in\R,\;z\in\R, \\
u(t,z) &= g_1(z+c_1 t), &t<0,\;z\in\R,
\end{aligned}
\end{equation}
for an initial wave $g_1$ with $\mbox{supp} g_1 \subset (z_1,\infty),$ for $n$ as in \eqref{single_ref}. 

This problem refers to the case of a wave incident from the right on the boundary $z=z_1$, see the right picture in Fig. \ref{fig_motion}. Again, we obtain a reflected and a transmitted part of the wave. The solution of this problem is given by
\begin{equation*}
u(t,z)=\begin{cases}
u_{-}(t,z)=g_0(z+ c_0 t), &t\geq 0, \, z<z_1,\\[1.5ex]
u_+ (t,z)=f_1(z- c_1 t)+g_1(z+ c_1 t), &t\geq 0, \,  z>z_1.
\end{cases}
\end{equation*} 
As previously, we find a representation of the reflected and transmitted waves using operators acting on the initial wave. Here, we denote the reflection operator by $R_-$ and the transmission operator by $T_-,$ having the forms 
\begin{equation*}
R_-: g_1 \mapsto f_1,
\quad R_-[g_1](z- c_1 t)=\frac{c_0-c_1}{c_1+c_0}g_1\left(2z_1-(z-c_1 t)\right),
\end{equation*}
and 
\begin{equation*}
T_-: g_1 \mapsto g_0, \quad 
T_-[g_1](z+c_0 t)=\frac{2c_0}{c_1+c_0}g_1\left(z_1+\frac{c_1}{c_0}((z+c_0 t)-z_1)\right).
\end{equation*}
We define the solution operator $U_2:g_1\mapsto u$, mapping the initial $g_1$ to the solution of the problem \eqref{l10}. 

It is trivial to model an operator $U_3 : f_1 \mapsto u,$ where $f_1$ satisfies  $\supp f_1 \subset (-\infty,z_2),$ for an interface at $z =z_2,$ with $n(z) = n_2,$ for $z > z_2.$ This setup models how a transmitted, from the boundary at $z =z_1,$ wave propagates for $t\geq 0.$ We know that on $(-\infty,z_2)$, $U_3[f_1]$ is of the form
\begin{equation*}
U_3[f_1](t,z)=f_1(z-c_1t)+g_1(z+c_1t).
\end{equation*} 
We define in addition the operator $R_+:f_1\mapsto g_1$.
These three different cases are combined to produce the following result.  
\begin{proposition}
\label{l15}
	Let $L = (z_1, \, z_2)$ be a single-layer medium, and let the refractive index be given by
	\begin{equation*}
     n(z)=\begin{cases}
     n_0,& z<z_1, \\
     n_1,& z\in (z_1,z_2), \\
     n_2,& z>z_2.
     \end{cases}
	\end{equation*} 
If the initial wave $f_0,$ see \eqref{l1}, satisfies $\supp f_0\subset (-\infty,z_1),$ then the solution of  \eqref{l53}, together with $u(t<0,z) = f_0,$ is given by
	\begin{equation}\label{l17}
	\begin{aligned}
	u(t,z) &=\b_{(-\infty,y)}\left(U_1[f_0](t,z)+ \sum_{j=0}^{\infty}{U_2\left[ (R_+R_-)^j R_+T f_0\right]  (t,z)}\right)\\
	&\phantom{=}+
	\b_{(y,\infty)}\sum_{j=0}^{\infty}{U_3\left[(R_-R_+)^j T f_0\right]}(t,z), \quad t\geq 0, \, y\in (z_1, \, z_2),	
	\end{aligned}
	\end{equation} 
	with $U_1, \, U_2,$ and $U_3$ defined as before. \end{proposition}  

\begin{proof}
	The function $u,$ given by \eqref{l17}, is by construction a solution to the wave equation problem in both $(-\infty,y)$ and $(y,\infty)$. Thus, we have only to check if both parts coincide in the interval $(z_1,z_2)$. To do so, we recall the definitions
	\begin{align*}
	U_1[f](t,z) &= T[f](z-c_1t), \\
	U_2[g](t,z) &= R_-[g](z-c_1t)+g(z+c_1t),\\
	U_3[f](t,z) &= f(z-c_1t)+R_+[f](z+c_1t),
	\end{align*}
	for $t\geq 0,$ and $z\in (z_1, \, z_2).$
	By plugging these formulas in \eqref{l17} we get for $(-\infty,y)$ the term
	\begin{equation*}
	T[f_0](z-c_1t)+\sum_{j=0}^{\infty}{R_-\left[(R_+R_-)^j R_+T f_0\right]}(z-c_1t)+\sum_{j=0}^{\infty}{(R_+R_-)^j R_+T \left [f_0\right]}(z+c_1t),
	\end{equation*} 
	and for $(y,\infty)$ the term
	\begin{equation*}
	\sum_{j=0}^{\infty}{(R_-R_+)^j T \left[f_0\right]}(z-c_1t)+\sum_{j=0}^{\infty}{R_+\left[(R_-R_+)^j T f_0\right]}(z+c_1t).
	\end{equation*}
	Under the assumption that both series are convergent, we see that the two terms coincide. 
The last thing to show is that \eqref{l17} also satisfies the initial condition 
    $u(t,z)=f_0(z-c_0t),$
    for all $t<0$ and $z\in \R$. This can be seen by considering the supports of the operators $U_1 , \, U_2$ and $U_3,$ as they are defined previously. 
\end{proof}

The formula \eqref{l17}, consists of three terms and accounts also for the multiple reflections occurring in the single-layer medium. Each term neglects either the boundary at $z_1,$ or the one at $z_2$. The operator $U_1$ maps the initial wave $f_0$ to the solution $u$, in the real line. Then, the transmitted wave $T[f_0]$ is traveling back and forth between $z_1$ and $z_2$, describing the multiple reflections, given by the field
\begin{equation}
\label{l18}
(R_-R_+)^j T \left [f_0\right ]. 
\end{equation}   
The operator $U_2$ now uses for every $j\in \N$ the reflection of \eqref{l18} at $z_2$ as initial and gives back a solution of the sub-problem \eqref{l10}. The last term models the wave interacting with the boundary at $z=z_2$, by application of $U_3,$ which uses \eqref{l18} as an initial function.

In the following example, we present the forms of the  single- and double-reflected wave from the boundary at $z = z_2,$ measured in $(-\infty,z_1).$
 
\begin{example}
	\label{l20}
	We know already that the reflected wave from the boundary $z = z_1,$ is given by \eqref{l7}. We present now the reflected waves $\varrho_{ \eta_r}, \, \eta_r =1,2$ in the interval $(-\infty,z_1),$ where $\eta_r$ counts for the numbers of the undergoing reflections, meaning
	\begin{equation*}
	\varrho_1(t,z)=T_- R_+T [f_0], \quad \text{and}\quad
	\varrho_2(t,z)=T_- R_+R_-R_+T [f_0].
	\end{equation*}
For $\eta_r =1,$	using the definition of the operator $T$ applied to $f_0$ we get
	\begin{equation*}
	\varrho_1(t,z)=\frac{2c_1}{c_1+c_0}T_- R_+f_0\left(z_1+\frac{c_0}{c_1}((z-c_1t)-z_1)\right). 
	\end{equation*}
The argument of $f_0$ is now a function of $z-c_1t$, and we can apply $R_+,$ resulting to 
	\begin{equation*}
    \varrho_1(t,z)=\frac{c_2-c_1}{c_2+c_1}\frac{2c_1}{c_1+c_0}T_- f_0\left(z_1+\frac{c_0}{c_1}(2z_2-(z+c_1t)-z_1)\right). 
	\end{equation*}
Thus, we have
	\begin{equation*}
	R_+T [f_0] (z-c_0t)=g(z+c_1t),
	\end{equation*}
	a function of $z+c_1t,$ where the operator $T_-$ can be to applied to give
	\begin{equation*}
	\varrho_1(t,z)=\frac{2c_0}{c_1+c_0}\frac{c_2-c_1}{c_2+c_1}\frac{2c_1}{c_1+c_0}f_0\left(z_1+\frac{c_0}{c_1}(2(z_2-z_1)-\frac{c_1}{c_0}((z+c_0t)-z_1))\right).
	\end{equation*}
Following the same procedure for $\eta_r =2$ now, we end up with the following form
	\begin{equation*}
	\begin{aligned}
	\varrho_2(t,z) &=\frac{2c_0}{c_1+c_0}\frac{c_0-c_1}{c_1+c_0}\left(\frac{c_2-c_1}{c_2+c_1}\right)^2\frac{2c_1}{c_1+c_0} \\
&\phantom{=}\times	f_0\left(z_1+\frac{c_0}{c_1}(4(z_2-z_1)-\frac{c_1}{c_0}((z+c_0t)-z_1))\right).
	\end{aligned}
	\end{equation*}
\end{example}

Now, we move to the case of a multi-layer medium. The solution will be derived using the above formulas and consider the problem layer-by-layer. We define $L_j = (z_j,\, z_{j+1}),$ for $j =1,...,N.$ The refractive index is given by \eqref{eq_refract} and we define  
\begin{equation}
\label{l23}
n_-(z)=\begin{cases}
n_0, & z < z_1,\\
n_1, & z > z_1,
\end{cases}, \quad n_+(z)=\begin{cases}
n_1, & z < z_2,\\
\tilde n_+(z), & z > z_2.
\end{cases}
\end{equation}
The parameter $\tilde n_+$ represents the refractive index in the remaining $N-1$ layers. 

Next we want to find a solution of \eqref{l53} for 
\[
n(z)=\begin{cases}
n_- (z), & z < y,\\
n_+ (z), & z > y,
\end{cases}
\]
and $y\in (z_1, \, z_2).$  The first case is exactly the same as the problem already discussed for the single-layer case, meaning that the application of the operators $U_1$ and $U_2$ is still valid. For the later case, we consider an initial wave $f$ with $\mbox{supp} f\subset(-\infty,z_2)$ and by $U_3[f]$ we denote the solution of this sub-problem. We know that in $(-\infty,z_2)$, $U_3[f]$ gives
\begin{equation*}
u(t,z)=f(z-c_1 t)+g(z+c_1 t).
\end{equation*}
We define an operator $R_+$ with $R_+[f]=g,$ corresponding to the multi-reflected light from the boundaries $z_2,\dots,z_{N+1}$, if we illuminate illumination by $f$. Then, we get the following result.
 
\begin{proposition}
	\label{l25}
	Let $n$ be defined by \eqref{eq_refract}, and $n_-$, $n_+$ as in \eqref{l23}. Then the solution of \eqref{l53}, together with $u(t<0,z) = f_0,$ is given by 
	\begin{equation}
	\begin{aligned}
	\label{l26}
	u(t,z) &=\b_{(-\infty,y)}\left(U_1[f_0](t,z)+ \sum_{j=0}^{\infty}{U_2\left[(R_+R_-)^j R_+T f_0\right](t,z)}\right) \\
&\phantom{=}+
	\b_{(y,\infty)}\sum_{j=0}^{\infty}{U_3\left[(R_-R_+)^j T f_0\right]}(t,z), \quad t\geq 0, \, y \in (z_1, \, z_2).	
	\end{aligned}
	\end{equation}
\end{proposition}

We know that the solution of \eqref{l53} in $(-\infty, z_1)$ admits the form
\begin{equation*}
u(t,z)=f_0(z-c_0 t)+g(z+c_0 t),
\end{equation*}
and we define an operator $\tilde R,$ through $\tilde R[f_0] = g$. 

\begin{proposition}
	\label{l27}
	Let the operators $R, R_-$ and  $T$ be defined as previously and $\tilde R[f_0]$ be given. Then, the following holds
    \begin{equation}\label{eq_lemma24}
    R_+[\tilde f]=\tilde{g}, 
    \end{equation}  
 where 
	\begin{equation}
	\label{l28}
    \tilde f=(T+R_-U_2^{-1}(\tilde R-U_1+I))[f_0], \quad \text{and}\quad	\tilde g=U_2^{-1}(\tilde R-U_1+I) [f_0].
	\end{equation}
\end{proposition}

\begin{proof}
From Proposition \ref{l25} we know that
	\begin{equation*}
	u(t,z)=U_1[f_0](t,z)+ \sum_{j=0}^{\infty}{U_2\left[(R_+R_-)^j R_+T f_0\right]}(t,z),\quad t\geq 0, z < z_1, 
	\end{equation*}
	describes a solution. It also holds that 
	\begin{equation*}
	u(t,z)=f_0(z-c_0t)+\tilde R[f_0](z+c_0t), \quad t\geq 0, z < z_1.
	\end{equation*}
Then, using $I f_0(t,x)=f_0(z-c_0t)$, we get 
	\begin{equation*}
	U_2^{-1}(\tilde R-U_1+I)[f_0]=(1-R_+R_-)^{-1}R_+T [f_0],
	\end{equation*}
which admits the equivalent form 
	\begin{equation*}
	U_2^{-1}(\tilde R-U_1+I)[f_0]-R_+R_-U_2^{-1}(\tilde R-U_1+I)[f_0]=R_+T [f_0].
	\end{equation*}
	This results to \eqref{eq_lemma24} for $\tilde f$ and $\tilde g,$ as in \eqref{l28}.
\end{proof}

\begin{remark}
The function $\tilde g$ describes the total amount of light which travels back from the remaining $N-1$ layers, meaning it considers all multiple reflections.
\end{remark}

\begin{lemma}
	\label{lem25}
	Let $L$ be a multi-layer medium consisting of $N\in \N$ layers, and let the refractive index be given by \eqref{eq_refract}.  Then, the solution of \eqref{l53}, together with $u(t<0,z) = f_0,$ can be computed layer-by-layer.
\end{lemma}   

\begin{proof}
	Starting with the first layer, we use $n$ defined in \eqref{eq_refract} and \eqref{l23} and we apply Proposition \ref{l25}. We thus obtain $\tilde f$ and $\tilde g,$ presented in Proposition \ref{l27}. Then, the function $\tilde f$ is the initial wave for the corresponding problem with parameter now given by
	\begin{equation*}
	\label{l30}
	n(z)=\begin{cases}
	n_1,& z<z_2, \\
	n_2,& z\in (z_2,z_3), \\
	\tilde n(z),& z>z_3,
	\end{cases}
	\end{equation*}
where $\tilde n$ represents the refractive index of the next $N-2$ layers.  Repeating the same argument, we use \eqref{l26}, with $f_0$ replaced by $\tilde f,$ for the updated operators. We continue this procedure for the new parameters and operators and we end up with the solution for $n$ given by  \eqref{eq_refract}.
\end{proof}

After some lengthy but straightforward calculations, we  can generalize the formulas of Example \ref{l20} for the $k$-th layer of the medium and $\eta_r \in \N,$ resulting to the field 
	\begin{equation}
	\label{multi}
	\begin{aligned}
	\varrho_{\eta_r}(t,z) &=\sum_{q=1}^{\eta_r}{\left(\frac{c_{k-1}-c_k}{c_{k}+c_{k-1}}\right)^{q-1}\left(\frac{c_{k+1}-c_k}{c_{k+1}+c_k}\right)^q}\prod_{j=1}^{k}{\frac{4c_{j-1}c_j}{({c_j}+c_{j-1})^2}}\\
&\phantom{=}\times	f_0\left( \left(\sum_{j=1}^{k-1}{2z_j(1-\frac{c_{j-1}}{c_{j}})\prod_{l=1}^{j-1}{\frac{c_{l-1}}{c_{l}}}}+z_{k}(2-(2+2(q-1))\frac{c_{k-1}}{c_{k}})\prod_{l=1}^{k-1}{\frac{c_{l-1}}{c_{l}}}\right)\right.
\\
&\left.\phantom{=}+z_{k+1}\frac{c_0}{c_{k}}(2+2(q-1))-(z+c_0t)\right),	\quad t\geq 0, \, z < z_1,
	\end{aligned}
\end{equation}
valid for an initial function $f_0,$ with $\mbox{supp} f_0\subset (-\infty, z_1).$

\subsection{Dispersive medium}

In this section, we consider the form \eqref{rei} for the refractive index and we set  $\hat n_0 (\omega) = \hat n_0 >0.$ We assume $\Re \{\hat n (\omega)\} >0,$ and 
$\Im \{\hat n (\omega)\} \geq 0,$ for all $\omega\in\R.$ We recall \eqref{eq_helm}.
Unfortunately,  an explicit solution, as in Sec. \ref{subsec_non}, cannot be derived here for a time dependent parameter. However, applying the Fourier transform with respect to time, we get the Helmholtz equation
\begin{equation}
\label{l32}
\Delta \hat u(\omega,z)+\frac{\omega^2}{c^2}\hat n^2(\omega,z)\hat u(\omega,z)=0, \quad \omega \in \R, \, z\in \R.  
\end{equation} 

For $y\in(z_1,\, z_2),$ we define 
\[ \hat n(\omega,z) = \begin{cases}\hat n_-(\omega,z), &z<y,\\
\hat n_+(\omega,z), &z>y,\end{cases} \]
with
\[ \hat n_-(\omega,z) = \begin{cases}\hat n_0, &z<z_{1},\\ 
\hat n_1(\omega), &z > z_1,\end{cases}\quad \text{and}\quad\hat n_+(\omega,z) = \begin{cases}\hat n_1(\omega), &z < z_2,\\
\tilde n(\omega,z), &z>z_{2}.\end{cases} \]
The refractive index $\tilde n$ accounts for the parameter of the remaining $N-1$ layers. 

Initially, we consider the problem of a right-going incident wave of the form
\begin{equation}\label{eq_inc_freq}
\hat u_0(\omega,z)=\alpha_0(\omega)e^{i\frac{\omega}{c} \hat n_0 z},
\end{equation}
incident at the interface $z = z_1.$ Then, the corresponding problem reads
\begin{equation*}
\begin{aligned}
\Delta \hat u(\omega,z)+\frac{\omega^2}{c^2}\hat n_-^2(\omega)\hat u(\omega,z) &=0, & \omega &\in \R, \, z\in\R,   \\
\partial_z\hat u-i\frac{\omega}{c}\hat n_1(\omega)\hat u &=0, & \omega &\in \R, \, z=z_1^+,
\end{aligned}
\end{equation*} 
for an artificial boundary point $z_1^+>z_1$. The boundary radiation condition is such that there is no left-going wave at the region $(z_1, +\infty).$

The solution admits the form
\begin{equation*}
\hat u(\omega,z)=\begin{cases}
\hat u_0+R[\alpha_0](\omega)e^{-i\frac \omega c \hat n_0  z},& z<z_1, \\[1.5ex]
T[\alpha_0](\omega)e^{i\frac \omega c \hat n_1(\omega) z}, & z>z_1,
\end{cases}
\end{equation*}
where we define the reflection and transmission operators $R$ and $T,$ respectively,  by 
\begin{equation}
\label{l39}
\begin{aligned}
R(\omega) &: \alpha_0(\omega)\mapsto \frac{\hat n_0 -\hat n_1(\omega)}{\hat n_0 +\hat n_1(\omega)}\alpha_0(\omega)e^{2i\frac \omega c \hat n_0 z_1}, \\
T(\omega) &: \alpha_0(\omega)\mapsto \frac{2\hat n_0 }{\hat n_0  + \hat n_1(\omega)}\alpha_0(\omega)e^{i\frac \omega c (\hat n_0-\hat n_1 (\omega)) z_1}.
\end{aligned}
\end{equation} 
The solution operator is then given by $V_1: \hat u_0\mapsto \hat u$.  The next sub-problem is described by
\begin{equation*}
\begin{aligned}
\Delta \hat u(\omega,z)+\frac{\omega^2}{c^2}\hat n_-^2(\omega)\hat u(\omega,z) &=0, & \omega &\in \R, \,z\in\R,  \\
\partial_z\hat u+i\frac{\omega}{c}\hat n_0 \hat u &=0, & \omega &\in \R, \,  z=z_1^- .
\end{aligned}
\end{equation*} 
for an incident left-going wave of the form $\hat u_{0}(\omega,z)=\beta_1(\omega)e^{-i\frac \omega c \hat n_1(\omega) z},$ and an artificial boundary point at $z_1^-< z_1$. The boundary radiation condition is that the left-going wave in $(-\infty, z_1)$ is zero. The solution now is given by
\begin{equation*}
\hat u(\omega,z)=\begin{cases}
 T_- [\beta_1](\omega)e^{-i\frac \omega c \hat n_0 z},& z<z_1, \\[1.5ex]
\hat u_0 + R_-[\beta_1](\omega)e^{i\frac \omega c \hat n_1(\omega) z}, & z>z_1,
\end{cases}
\end{equation*}
where $R_-$ and $T_-$ are defined by
\begin{equation*}
\begin{aligned}
R_-(\omega) &: \beta_1(\omega)\mapsto \frac{\hat n_1(\omega)-\hat n_0}{\hat n_1(\omega)+\hat n_0}\beta_1(\omega)e^{2i\frac \omega c \hat n_1(\omega) z_1}, \\
 T_- (\omega) &: \beta_1(\omega)\mapsto \frac{2\hat n_1(\omega)}{\hat n_1(\omega)+\hat n_0 }\beta_1(\omega)e^{i\frac \omega c (\hat n_0-\hat n_1)(\omega) z_1}.
\end{aligned}
\end{equation*}  
Let again $V_2:\hat u_{0} \mapsto \hat u,$ denote the corresponding solution operator. 

The final sub-problem, deals with the scattering of a right-going wave of the form   $\hat u_{0}(\omega,z)=\alpha_1(\omega)e^{i\frac{\omega}{c} \hat n_1(\omega) z},$ by a medium supported in  $(z_2, \, +\infty)$ with refractive index $\tilde{n}.$ The governing equations are 
\begin{equation*}
\begin{aligned}
\Delta \hat u(\omega,z)+\frac{\omega^2}{c^2}\hat n_+^2(\omega)\hat u(\omega,z) &= 0, & \omega &\in \R, \, z\in\R, \  \\
\lim_{z\rightarrow +\infty}\hat u  (\omega,z) &= e^{i\frac{\omega}{c}\int_{z_2}^{z}\tilde n(\omega, y)d y }, & \omega &\in \R.
\end{aligned}
\end{equation*}
The radiation condition now ensures that in $(z_2 , \, +\infty
)$ exist only right-going waves.
The solution is given by
\begin{equation*}
\hat u(\omega,z)=\alpha_1(\omega)e^{i\frac{\omega}{c} \hat n_1(\omega) z}+ \beta_1(\omega)e^{-i\frac{\omega}{c} \hat n_1(\omega) z}, \quad z<z_2.
\end{equation*} 
We define $R_+:\alpha_1(\omega)\mapsto \beta_1(\omega),$ and the relevant operator $V_3: \hat u_{0}\mapsto \hat u,$ mapping the incident field to the solution of this specific problem. 
We remark that the operator $R_+$  cannot be computed explicitly because it contains also the information from the remaining $N-1$ layers.

 \begin{proposition}
	\label{lemma01}
	Let the incident wave be of the form \eqref{eq_inc_freq}.  We define
	\begin{equation*}
	\begin{aligned}
	\hat u^j_{0,-}(\omega,z) &= [(R_+R_-)^j R_+T\alpha_0](\omega)e^{-i\frac \omega c \hat n_1(\omega) z}, \\
	\hat u^j_{0,+}(\omega,z) &= [(R_-R_+)^j T\alpha_0](\omega)e^{i\frac \omega c \hat n_1(\omega) z}. 
	\end{aligned}
	\end{equation*}
	Then, the field
	\begin{equation}
	\label{sol}
	\hat u(\omega,z)=\b_{(-\infty,y)}\left( V_1[\hat u_0]+\sum_{j=0}^{\infty}{V_2 [\hat u^j_{0,-}]}\right) (\omega,z) +\b_{(y,\infty)}\left(\sum_{j=0}^{\infty}{V_3 [\hat u^j_{0,+}]}\right)(\omega,z),
	\end{equation}
	for $y\in (z_1, z_2),$ is the solution of the Helmholtz equation \eqref{l32}, for the refractive index defined as above. 
\end{proposition}

\begin{proof}
	By construction, $\hat u$ fulfills \eqref{l32} in $(-\infty,y)$ and $(y,\infty)$. Thus, it remains to show that the two parts coincide in $(z_1,z_2).$
	Recalling the definitions of $V_1,\, V_2,$ and $V_3,$ restricted in $(z_1,z_2)$, we get 
\begin{multline*}
V_1[\hat u_0](\omega,z) +\sum_{j=0}^{\infty} V_2 [\hat u^j_{0,-}](\omega,z)  =  
    T[\alpha_0](\omega)e^{i\frac \omega c \hat n_1(\omega) z} \\
    +\sum_{j=0}^{\infty}{R_- [(R_+R_-)^j R_+T\alpha_0](\omega)}e^{i\frac \omega c \hat n_1(\omega) z} 
    +\sum_{j=0}^{\infty}{(R_+R_-)^j R_+T[\alpha_0](\omega)}e^{-i\frac \omega c \hat n_1(\omega) z}
\end{multline*}
and
\begin{align*}
	  \sum_{j=0}^{\infty}{V_3 [\hat u^j_{0,+}]}(\omega,z) &=
     \sum_{j=0}^{\infty}{(R_-R_+)^j T[\alpha_0](\omega)}e^{i\frac \omega c \hat n_1(\omega) z} \\
       &\phantom{=}+\sum_{j=0}^{\infty}{R_+ [(R_-R_+)^j T\alpha_0](\omega)}e^{-i\frac \omega c \hat n_1(\omega) z}.
\end{align*}
	We reorder the terms and we observe that they coincide in $(z_1,z_2).$
\end{proof}

\begin{remark}
If $L=(z_1,z_2)$ denotes a single-layer medium, with material parameter $\hat n$, given by
		\begin{equation*}
		\hat n(\omega,z)=\begin{cases}
		\hat n_0, & z<z_1,\\
		\hat n_1(\omega), & z\in(z_1,z_2), \\
		\hat n_2(\omega), & z>z_2,
		\end{cases}
		\end{equation*}
		then we can compute $R_+$ explicitly, and also the operator $V_3$. 
\end{remark}

\begin{example}
The amplitude of the $j$-th reflection in a certain layer is described by the term 
		\begin{equation*}
		(R_+ R_-)^j R_+ T[\alpha_0] (\omega).
		\end{equation*}
		The single reflected wave from the most left boundary of $L$ is given by \eqref{l39}. For the $k$-th layer of the medium, we obtain the back-reflected field 
		\begin{equation}\label{l43}
		\begin{aligned}
		\hat \varrho (\omega, z) &= \sum_{q=1}^{\eta_r}\Big (\frac{\hat n_k(\omega)-\hat n_{k+1}(\omega)}{\hat n_{k+1}(\omega)+\hat n_k(\omega)}\Big)^q \Big (\frac{\hat n_{k}(\omega)-\hat n_{k-1}(\omega)}{\hat n_{k-1}(\omega)+\hat n_k(\omega)}\Big)^{q-1}	\prod_{j=1}^{k}{\frac{4\hat n_{j-1}(\omega)\hat n_j(\omega)}{(\hat n_{j-1}(\omega)+\hat n_j(\omega))^2}}\\
		&\phantom{=}\times \alpha_0(\omega)e^{i\frac{\omega}{c}\left(\hat n_k(\omega)(2q z_{k+1}-2(q-1)z_k)+\sum_{l=1}^{k}{2(\hat n_{l-1}-\hat n_{l})(\omega)z_l}\right)}e^{-i\frac{\omega}{c}\hat n_0 z}, \quad z<z_1, 
		\end{aligned}
		\end{equation}
		where $\alpha_0$ is the amplitude of the incident wave $\hat u_0$.
\end{example}

The solution of \eqref{l32} in $(-\infty,z_1)$ admits the form 
\begin{equation*}
\hat u(\omega,z)=\alpha_0(\omega)e^{i\frac{\omega}{c}\hat n_0 z}+\tilde \beta(\omega)e^{-i\frac{\omega}{c}\hat n_0 z},
\end{equation*}
and we define $\tilde R(\omega):\alpha_0(\omega)\mapsto \tilde \beta(\omega)$.

\begin{lemma} 
Let the incident wave be of the form \eqref{eq_inc_freq}, and let $\tilde R[\alpha_0]$ be known. Then, the following relation holds  
	\begin{equation*}
	R_+[\tilde \alpha_0](\omega)=\tilde \beta_0(\omega),
	\end{equation*}
	for
	\begin{equation}
	\label{amp}
	\tilde \alpha_0=T[\alpha_0]+R_-\left( T_-^{-1}\left(\tilde R[\alpha_0]-R[\alpha_0]\right)\right), \quad\text{and}\quad \tilde \beta_0= T_-^{-1}\left(\tilde R[\alpha_0]-R[\alpha_0]\right),
	\end{equation}
	calculated from the previously defined operators $R_-$, $T$ and $T_-$.  
\end{lemma}

\begin{proof}
	We know that in $(-\infty,z_1)$, 
	\begin{equation*}
	\alpha_0(\omega)e^{i\frac{\omega}{c}\hat n_0 z}+\tilde R[\alpha_0](\omega) e^{-i\frac{\omega}{c}\hat n_0 z}= V_1[\hat u_0](\omega,z)+\sum_{j=0}^{\infty}{V_2 [\hat u^j_{0,-}]}(\omega,z),
	\end{equation*}
for $\hat u^j_{0,-},$ defined as in Proposition \ref{lemma01}.
	Using the definitions of $V_1$ and $V_2,$ we get
	\begin{equation*}
\tilde R[\alpha_0](\omega) e^{-i\frac{\omega}{c}\hat n_0 z} -R[\alpha_0](\omega) e^{-i\frac{\omega}{c}\hat n_0  z}=
	T_- \left(\sum_{j=0}^{\infty}{ (R_+R_-)^jR_+T[\alpha_0]}\right )(\omega) e^{-i\frac{\omega}{c}\hat n_0 z}.
	\end{equation*}
This results to 	
	\begin{equation*}
     T_-^{-1} \left(\tilde R[\alpha_0] -R[\alpha_0] \right)=\sum_{j=0}^{\infty}{ (R_+R_-)^jR_+T[\alpha_0]},
	\end{equation*}
which is equivalent to
	\begin{equation*}
	(1-R_+R_-)\left( T_-^{-1} \left(\tilde R[\alpha_0] -R[\alpha_0] \right)\right)=R_+T[\alpha_0]. 
	\end{equation*}
	This completes the proof.
\end{proof}

The amplitudes $\tilde \alpha_0$ and $\tilde \beta_0$, defined in \eqref{amp}, correspond to the amplitudes of the Fourier transforms of $\tilde f$ and $\tilde g,$ given in Proposition \ref{l27}. Furthermore, one can derive an analogue of Lemma \ref{lem25} also for a dispersive medium.

\section{The inverse problem}\label{sec3}

 We address the inverse problem of recovering the position, the size and the optical properties of a multi-layer medium with piece-wise inhomogeneous refractive index. We identify the position by the distance from the detector to the most left boundary of the medium, and the size by reconstructing the constant refractive index $n_0$ of the background medium. Initially, we discuss the problem of phase retrieval and possible directions to overcome it and then we present reconstruction methods for non-dispersive and dispersive media. We end this section by giving a method, which with the use of the Kramers-Kronig relations, makes the reconstruction of a complex-valued refractive index (absorbing medium) possible. Let $z= z_d$ denote the position of the point detector. 
 
 \subsection{Phase retrieval and OCT}\label{subsec_phase}

The phase retrieval problem, meaning the reconstruction of a function from the magnitude of its Fourier transform, has attracted much attention in the optical imaging community, see \cite{SheEldCOhCha15} for an overview. When dealing with experimental data, additional problems arise, like different types of noise and incomplete data. Mathematically speaking, the problem corresponds to a least squares minimization problem for a non-convex functional. In our case, where we are given one-dimensional data of the form \eqref{eq_fourier_data} or \eqref{eq_fourier_data2}, unique reconstruction of the phase is not possible \cite{Hof64}. However, there exist convergent algorithms that produce satisfactory results under some assumptions on the signal, like bounded support and non-negativity constraints. These algorithms are alternating between time and frequency domains, using usually  less coefficients than samples,  which makes the exact recovery almost impossible. 

In OCT, this problem has been also well studied, see for example \cite{LeiHitFerBaj03, MukSee12, SeeVilLeiUns08}. The main idea is either to consider a phase-shifting device in the reference arm or to combine OCT with holographic techniques. The first case, the one we consider here, produces different measurements by placing the mirror at different positions, meaning by changing the path-length difference between the two arms.

As already discussed in Sec. \ref{sec1}, we have measurements of the form
\begin{equation*}
\hat m (r; \omega ) = | (\hat u - \hat u_0) (\omega, z_d)  +  (\hat u_r - \hat u_0) (\omega, z_d) |, \quad  \omega\in\R,
\end{equation*}
for $r$ fixed, where $u_r$ denotes the $y-$component of the reference field $E_r,$ and we also acquire the data
\begin{equation*}
\hat m_s (\omega ) = | (\hat u - \hat u_0) (\omega, z_d) |, \quad  \omega\in \R.
\end{equation*}
Since we know the incident field $\hat u_0$ explicitly, we can also compute the reference field $\hat u_r - \hat u_0$ at the point detector. Then, the problem of phase retrieval we address here is to recover $\hat u - \hat u_0$ from the knowledge of $\hat m$ and $\hat m_s$ for all $\omega\in\R. $ We know, from \cite{KliKam14, KliSacTik95}, that if $u_r - u_0$ is compactly supported, then there exist at most two solutions $u-u_0$. See the left picture of Fig. \ref{fig_phase}, where we visualize graphically the two solutions by plotting in the complex plane the two above equations at specific frequency for the setup of the third example presented later in Sec. \ref{sec4}.

If in addition, there exist a constant $\gamma \in [0,1),$ such that
\[
| \Re \{\hat u_r - \hat u_0 \} | \leq \gamma | \Im \{\hat u_r - \hat u_0 \} |,
\]
then, there exists  at  most one  solution  in  $L^2(\R)$ with  compact  support  in  $[0,\infty ).$  However, it is hard to verify that the reference field fulfills this condition and we observed that, in all numerical examples, 
this inequality does not hold, for an incident plane wave.  Thus, in order to decide which solution of the two is correct, extra information is needed. Motivated by the phase-shifting procedure, we consider data for two different positions of the mirror, let us say $r=r_1, r_2.$ Then, we get the data
\begin{equation*}
\hat m (r_1 ; \omega ), \quad  \hat m (r_2 ; \omega ), \quad \mbox{and} \quad \hat m_s (\omega ) .
\end{equation*}
 Using $\hat m (r_1 ; \cdot )$ and $\hat m_s,$ we get two possible solutions, and from $\hat m (r_2 ; \cdot )$ and $\hat m_s,$ other two. But since $\hat m_s$ is the same in both cases, we find the unique solution as the common solution  of the two pairs. This is illustrated at the right picture of Fig. \ref{fig_phase}, where we plot the three above relations at two different frequencies. This way, we get unique solutions at every available frequency. 
 
 Thus, having measurements for two different positions of the reference mirror, we may consider that frequency-domain OCT provide us with the quantity 
 \begin{equation}
 \label{freqm}
( \hat u - \hat u_0) (\omega, z_d) , \quad  \omega\in \R,
 \end{equation}
 the equivalent measurements of a time-domain OCT system, see \eqref{data_oct_time}.
 
 \begin{figure}[t]
    \centering
  \includegraphics[width=0.9\textwidth]{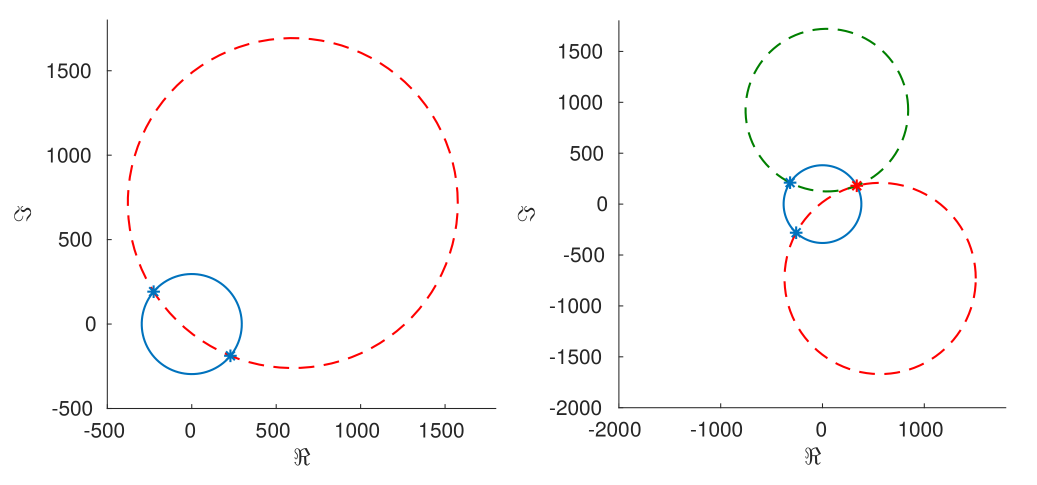}
    \caption{ Left: The intersection points of the  circles $\hat m (r_1 ; \omega_1) \in \C$ (red) and $\hat m_s (\omega_1) \in \C$ (blue). Right: The intersection points of the  circles $\hat m (r_1 ; \omega_2) \in \C$ (red),  $\hat m (r_2 ; \omega_2) \in \C$ (green) and $\hat m_s (\omega_2) \in \C$ (blue). The red asterisk indicates the unique solution. The setup is the same as the one in the third example in Sec. \ref{sec4}.
        }\label{fig_phase}
\end{figure}

\subsection{Reconstructions in time domain}\label{subsec_time}
We consider initially a non-dispersive medium.  Then, the refractive index is given by \eqref{eq_refract} and the  time-dependent OCT data admit the form
\begin{equation}
\label{l48}
 m(t,z_d) =u(t,z_d)-f_0(z_d-c_0t), \quad t\in[0,T],
\end{equation} 
for $T>0.$ Here, we assume that the initial wave (known explicitly) does not contribute to the measurements. The following presented algorithms are based on a layer-by-layer procedure. At the first step, we do reconstruct the parameters for a given layer and then we update the data, to be used for the next layer.
Thus, we assume that we have already recovered the boundary point $z_{k-1}$ and the coefficient $c_{k-1}$ of the layer $L_{k-1}.$ We denote by $m^{(k)} (t, z_d),$ the data corresponding to a $(N-k+1)-$layer medium, with the most left layer being the $L_{k}.$ 

As we see in Fig. \ref{fig_real}, the (time-dependent) data consist mainly of $N+1$ major $``$peaks$"$, for a $N-$layer medium, and some minor $``$peaks$"$,  related to the light undergoing multiple-reflections in the medium. The experimental data are, of course, also noisy and may show some small $``$peaks$"$ because of the OCT system. The first two $``$peaks$"$ correspond, for sure, to the single-reflected light from the first two boundaries. We propose a scheme to neglect minor $``$peaks$"$ due to multiple reflected light. Then, the first major $``$peak$"$ in $m^{(k)},$  corresponds to the back-reflected light from the interface at $z_k$. 

\begin{description}
\item[Step 1:] We isolate the first $``$peak$"$ by cutting off around a certain time interval $[T_1,T_2],$ meaning we consider
\begin{equation*}
\tilde{m}^{(k)} (t,z_d) =\b_{[T_1,T_2]}(t) m^{(k)}(t,z_d). 
\end{equation*}
The time interval can be fixed for all layers and depends on the time support of the initial wave.  On the other hand, this wave can be described by \eqref{multi}, if we use $\eta_r =1$ and $k:=k-1.$ Therefore, we obtain the equation
\begin{equation}
\label{mk}
\tilde{m}^{(k)} (t,z_d)=\frac{c_{k}-c_{k-1}}{c_{k}+c_{k-1}}\tilde f_0 \left(2z_k\frac{c_0}{c_{k-1}}-(z_d+c_0t) \right), \quad t\in\R,
\end{equation}
with $\tilde f_0$ given by
\begin{equation*}
\begin{aligned}
\tilde f_0 (z) &=\prod_{j=1}^{k-1}{\frac{4c_{j-1}c_j}{({c_j}+c_{j-1})^2}} 
f_0\left( \left(\sum_{j=1}^{k-2}{2z_j(1-\frac{c_{j-1}}{c_{j}})\prod_{l=1}^{j-1}{\frac{c_{l-1}}{c_{l}}}} \right.\right.\\
&\phantom{=}\left.\left.+z_{k-1}(2-2\frac{c_{k-2}}{c_{k-1}})\prod_{l=1}^{k-2}{\frac{c_{l-1}}{c_{l}}}\right)+z\right).
\end{aligned}
\end{equation*}
The supremum of \eqref{mk}, using its shift-invariance property, gives the value of $c_k$. The position of the interface $z_k$ can then be recovered from \eqref{mk}, by solving the minimization problem
\begin{equation*}
\min_{z\in \R}\Big \vert\tilde{m}^{(k)}(t,z_d)-\frac{c_k-c_{k-1}}{c_k+c_{k-1}}\tilde f_0(2z\frac{c_0}{c_{k-1}}-(z_d+c_0t)) \Big \vert, \quad \forall t \in  \R.
\end{equation*}
Since both parameters are time-independent, there exist also other variants for solving this overdetermined problem. 

\item[Step 2:]  Before moving to the layer $L_{k+1}$, we have to update the data function. We could just remove the contribution of the current layer, meaning $\tilde{m}^{(k)}.$
However, the first $``$peak$"$ might not correspond to the reflection from $z_k,$ for $k>1,$ but to contributions of multiple reflected wave from previous layers, arriving at the detector before the major wave. Since, we have recovered the properties of $L_k,$ we can compute all future multiple reflections from this layer using \eqref{multi},  let us call them $\mathcal{R}[c_k, \, z_k].$ 

Then, we update the data as
 \begin{equation*}
m^{(k+1)} =m^{(k)} -\tilde{m}^{(k)} - \mathcal{R}[c_k, \, z_k].
\end{equation*} 
\end{description}

Repeating these steps, we end up with the following result.

\begin{lemma}
	\label{l50}
	Let $L$ be a multi-layer medium, with $N\in \N$ layers, characterized by  $n,$ given by \eqref{eq_refract}. Then from the knowledge of $n_0$, the initial wave $f_0,$ and the measurement data \eqref{l48}, following the above iterative scheme, we can uniquely reconstruct $n_j$ and $z_j$ for $j =1,...,N+1.$   
\end{lemma}  

The above scheme can be written in an operator form, by the application of Propositions \ref{l25} and \ref{l27}.  For the sake of presentation, we consider the case of the first layer in order to avoid redefining all operators. 
 
 \begin{description}
 \item[Step 1:] Recall the definition of the operator $\tilde R$,  applied to the initial  wave $f_0,$ which describes the total amount of the reflected light. Considering the data \eqref{l48}, we get
\begin{equation*}
m (t,z_d)=\tilde R [f_0](z_d+c_0t) , \quad t\in [0,T].
\end{equation*} 
Following the same procedure, using \eqref{l7}, we can recover $c_1$ and $z_1,$ from the reduced data equation
\[
 m^{(1)} (t,z_d)= R [f_0](z_d+c_0t).
\] 

 \item[Step 2:]  We update all operators and from Proposition \ref{l27}, we obtain $\tilde f$ and $\tilde g$, given by \eqref{l28}. From the definition of $\tilde g$, we see that we obtain the function
	\begin{equation*}
	m^{(2)}(t,z_d) = m (t,z_d)-  m^{(1)}(t,z_d),
	\end{equation*} 
describing the updated data.  The advantage here is that we do not need to subtract the multiple reflections term, since they are already included in the updated version of $\tilde g.$

 \end{description}

\subsection{Reconstructions in frequency domain}\label{subsec_dis}

We consider the case of a dispersive medium, with a piece-wise inhomogeneous refractive index  $\hat n(\omega), $ for $\omega\in \R.$  Initially, we assume $\hat n \in \R.$ The derived iterative scheme can be applied also to the simpler non-dispersive case, giving a reconstruction method in the frequency domain for a time-independent parameter. 

\subsubsection{Dispersive medium}

The refractive index is given by \eqref{rei} and initially, we restrict ourselves to the case of real-valued $\hat n.$ The data $\hat m$ are given by \eqref{freqm}. As before, we present a layer-by-layer scheme. We assume that  the boundary point $z_{k-1}$ and the coefficient $\hat n_{k-1}$ of the layer $L_{k-1}$ are already recovered. We denote by $\hat m^{(k)}(\omega, z_d)$, the data corresponding to the $(N-k+1)-$layer medium.  Here, we need to choose the time interval $\b_{[T_1,T_2]}$ wisely, compared to Sec. \ref{subsec_time}, because of the presence of dispersion. However, this results only to slightly longer time interval, since in the wavelength range, where OCT operates, scattering dominates absorption.

\begin{description}
	\item[Step 1:] As we have seen in Fig. \ref{fig_real}, for example,  from the data 
$ \hat m^{(k)} $ we cannot distinguish the different $``$peaks$".$ Thus, we have to switch back to the time domain in order to isolate the first $``$peak$".$  We apply 
	\begin{equation*}
	\tilde{\hat{m}}^{(k)}(\omega)=\mathcal F \left (\b_{[T_1,T_2]}\mathcal{F}^{-1}(\hat m^{(k)})\right )(\omega), \quad \omega\in \R.
	\end{equation*}
	We use \eqref{l43} for $k:=k-1$ and $\eta_r=1$ to get 
	\begin{equation}
	\label{m_new}
	\tilde{\hat{m}}^{(k)}(\omega)=\frac{\hat n_{k-1}(\omega)-\hat n_{k}(\omega)}{\hat n_{k-1}(\omega)+\hat n_{k}(\omega)} \tilde \alpha_0(\omega) e^{i\frac{\omega}{c}\hat n_{k-1}(\omega)2 z_k}e^{-i\frac{\omega}{c}\hat n_0 z_d},
	\end{equation}	
	with
	\begin{equation*}
	\tilde \alpha_0(\omega)=\prod_{j=1}^{k-1}{\frac{4\hat n_{j-1}(\omega) \hat n_j (\omega)}{(\hat n_{j-1}(\omega)+\hat n_j (\omega))^2}} \alpha_0(\omega)e^{i\frac{\omega}{c}\sum_{l=1}^{k-1}{2(\hat n_{l-1}-\hat n_{l})(\omega)z_l}},
	\end{equation*}
	describing an already known quantity. The absolute value of \eqref{m_new} gives 
	\begin{equation*}
	\hat n_k(\omega)=\hat n_{k-1} (\omega) \left(\frac{1+ \vert \tilde{\hat{m}}^{(k)}(\omega) / \tilde \alpha_0(\omega)\vert}{1-\vert \tilde{\hat{m}}^{(k)}(\omega) / \tilde \alpha_0(\omega)\vert}\right)^{\pm 1}, 
	\end{equation*}
	which together with a suitable condition on $\hat n_k$ allow us to recover the refractive index.
	
To reconstruct the position $z_k$ we define the function
\[
f_k (\omega) = \frac{\tilde{\hat{m}}^{(k)}(\omega) / \tilde \alpha_0(\omega)}{\vert \tilde{\hat{m}}^{(k)}(\omega) / \tilde \alpha_0(\omega)\vert} \sign (\hat n_{k-1}(\omega)-\hat n_{k}(\omega)),
\]
and we observe that the absolute value of its derivative, together with \eqref{m_new}, results to
 	\begin{equation*}
	c \vert f_k'(\omega)\vert  = \vert  2 z_k (\hat n_{k-1}(\omega) + \omega \hat n'_{k-1}(\omega) ) - \hat n_0 z_d \vert.
	\end{equation*} 
	Together with a non-negativity constrain, we obtain $z_k.$
	\item[Step 2:]  As in the time domain case, we update the data by subtracting \eqref{m_new} and the terms representing the multiple reflections from the already recovered layers, called $\mathcal{R}[\hat n_k,z_k].$ We define 
	\begin{equation*}
	\hat{m}^{(k+1)}=\hat{m}^{(k)} -\tilde{\hat{m}}^{(k)}- \mathcal{R}[\hat n_k,z_k].
	\end{equation*} 
\end{description}  

Repeating the steps, a reconstruction of the properties and the lengths of all the remaining layers is obtained. 

\begin{lemma}
	Let $L$ be a multi-layer medium, with $N\in \N$ layers, characterized by $\hat n$, given by the Fourier transform of \eqref{rei}. If we restrict ourselves to the case $\hat n(\omega)\in \R,$ for all $\omega\in\R$, then the above iterative scheme, allows us to uniquely reconstruct $\hat n_j$ and $z_j$ for $j =1,...,N+1,$ given $\hat n_0$, the incident wave $\hat u_0$ and the measurement data \eqref{freqm}. 
	\end{lemma}

\subsubsection{Absorbing medium}
Here, we consider the case of a complex-valued material parameter, $\hat n(\omega)\in \C,$ for every $\omega\in \R.$ The real part describes how the medium reflects the light and the imaginary part (wavelength dependent) determines how the light absorbs in the medium. In \cite{ElbMinSch17, ElbMinSch18} we considered the multi-modal PAT/OCT system, meaning that we had additional internal data from PAT, in order to recover both parts of the refractive index. Here, the multi-layer structure allows us to derive an iterative method that requires only OCT data.  We decompose $\hat n$ as
\[
\hat n(\omega) = \nu (\omega) + i \kappa (\omega), \quad \nu, \, \kappa \in \R.
\]
The measurements are again given by \eqref{freqm}. The above presented iterative scheme, fails in this case.  Indeed, recall the formula \eqref{m_new}, which we considered for recovering the parameters of the $k-$th layer.  Taking the absolute value, for $\hat n_{k-1} (\omega), \, \hat n_k (\omega) \in \C,$ gives
\[
|\tilde{\hat{m}}^{(k)}(\omega)| = \frac{\vert \hat n_{k-1}(\omega)-\hat n_k(\omega)\vert}{\vert \hat n_{k-1}(\omega)+\hat n_k (\omega) \vert} \vert  \tilde \alpha_0(\omega)\vert e^{-\frac{\omega}c \kappa_{k-1} (\omega)2 z_k}.
\]
The last term in the above  expression, describes how the amplitude of the wave decreases, i.e. attenuation, and prevents us from a step-by-step solution, since $z_k$ still appears. Thus, we propose a different scheme that takes into account also the relations between the parts $\nu$ and $\kappa,$ meaning the Kramers-Kronig relations. We stress here that $\hat n$ is holomorphic in the upper complex plane, satisfying $\hat n (\omega) = \hat n(-\omega)^\ast. $ The parts of the complex-valued refractive index are connected through
\begin{equation}\label{eq_kk}
\begin{aligned}
\nu (\omega) -1 &=\frac{2}{\pi}\int_{0}^{\infty}\frac{\omega' \kappa (\omega')}{\omega'^2-\omega^2}d\omega', \\
\kappa (\omega)  &=-\frac{2\omega }{\pi}\int_{0}^{\infty}\frac{\nu (\omega') -1}{\omega'^2-\omega^2}d\omega'. \\
\end{aligned}
\end{equation}

In addition, defining the reflection coefficient,
\begin{equation}\label{eq_refle}
\rho_k (\omega) = \frac{\hat n_{k-1}(\omega)-\hat n_{k}(\omega)}{\hat n_{k-1}(\omega)+\hat n_{k}(\omega)} \in \C,
\end{equation}
and using its expression in polar coordinates $\rho_k = |\rho_k | e^{i \theta_k},$ we obtain 
\[
\ln (\rho_k (\omega)) = \ln (|\rho_k (\omega)|) + i \theta_k (\omega),
\]
a function that diverges logarithmically as $\omega \rightarrow \infty,$ and is not square-integrable \cite{LucSaaPeiVar05}. However, the following relation holds for the phase of the complex-valued reflectivity
\begin{equation}\label{eq_kk2}
\theta_k (\omega)  =-\frac{2\omega}{\pi}\int_{0}^{\infty}\frac{\ln (|\rho_k (\omega')|)}{\omega'^2-\omega^2}d\omega' .
\end{equation}
We define the operator
\[
\mathcal{H}(\omega) : f \mapsto  -\frac{2\omega}{\pi}\int_0^\infty \frac{f (\omega')}{\omega'^2-\omega^2}d\omega' ,
\]
and we get the relations in compact form
\begin{equation*}
\kappa_k (\omega) = \mathcal{H}(\nu_k -1 ) (\omega) , \quad \mbox{and} \quad \theta_k (\omega)  = \mathcal{H}(\ln |\rho_k |) (\omega).
\end{equation*}

Of course, when working with the Kramers-Kronig relations \eqref{eq_kk} and \eqref{eq_kk2}, one has to deal with the problem that they assume that information is available for the whole spectrum, something that is not true for experimental data. Another problem could be the existence of zero's of the reflection coefficient in the half plane. There exist generalizations of there formulas that can overcome these problems, like the subtractive relations which require few additional data. We refer to \cite{GroOff91, HorArtLa15, PalWilBud98, You77} for works dealing with the applicability  and variants of the Kramers-Kronig relations. This practical problem is  out of the scope of this paper and will be considered in future work, where we will examine numerically, with simulated and real data, the validity of the proposed scheme.

\begin{description}
\item[Step 1:] At first, we consider the reconstruction of the interface $z_k.$ We apply once the logarithm to the absolute value of \eqref{m_new} and then we take imaginary part of the logarithm of \eqref{m_new}. 
Using the definition \eqref{eq_refle}, we obtain the system of equations
\begin{equation}\label{eq_system1}
\begin{aligned}
\ln(\vert \tilde{\hat{m}}^{(k)}(\omega) \vert) &= \ln (\vert \rho_k (\omega)\vert) + \ln (\tilde \alpha_0(\omega)) -2\tfrac{\omega}{c} \kappa_{k-1} (\omega) z_k,\\
\Im \{\ln( \tilde{\hat m}^{(k)}(\omega))\} &= \theta_k (\omega)-\tfrac{\omega}{c} \hat n_0 z_d +2\tfrac{\omega}{c}\nu_{k-1} (\omega) z_k .
\end{aligned}
\end{equation}
We define the data functions
\begin{align*}
\hat m^{(k)}_1 (\omega)  &:= \ln(\vert \tilde{\hat{m}}^{(k)}(\omega) \vert) - \ln (\tilde \alpha_0(\omega)), \\ \hat m^{(k)}_2 (\omega) &:= \Im \{\ln( \tilde{\hat m}^{(k)}(\omega))\} +\tfrac{\omega}{c} \hat n_0 z_d,
\end{align*}
and the system \eqref{eq_system1} takes the form
\begin{equation*}
\begin{aligned}
\ln (\vert \rho_k (\omega)\vert) -2\tfrac{\omega}{c} \kappa_{k-1} (\omega) z_k &= \hat m^{(k)}_1 (\omega),\\
\theta_k (\omega) +2\tfrac{\omega}{c}(\nu_{k-1} (\omega) -1) z_k  + 2\tfrac{\omega}{c} z_k &= \hat m^{(k)}_2 (\omega),
\end{aligned}
\end{equation*}
to be solved for $z_k.$ We rewrite the last equation using the formulas \eqref{eq_kk} and \eqref{eq_kk2} and the first equation, to obtain
\begin{equation*}
\begin{aligned}
\hat m^{(k)}_2 (\omega) &= 2\tfrac{\omega}{c} z_k  -\frac{2\omega}{\pi}\int_{0}^{\infty}\frac{\ln (|\rho_k (\omega')|)}{\omega'^2-\omega^2}d\omega'  +\frac{4\omega z_k}{\pi c}  \int_{0}^{\infty}\frac{\omega' \kappa_{k-1} (\omega')}{\omega'^2-\omega^2}d\omega'  \\
&= 2\tfrac{\omega}{c} z_k - \frac{2\omega}{\pi}\int_{0}^{\infty} \frac{\ln (\vert \rho_k (\omega')\vert) -2\tfrac{\omega'}{c} \kappa_{k-1} (\omega') z_k}{\omega'^2-\omega^2} d\omega' \\
&= 2\tfrac{\omega}{c} z_k + \mathcal{H} (\hat m_1^{(k)}) (\omega).
\end{aligned}
\end{equation*}
The last equation is solved at given frequency $\omega^\ast \neq 0,$ in order to obtain  the location of the interface
\[
z_k = \frac{c}{2\omega^\ast} \left(\hat m^{(k)}_2 (\omega^\ast) - \mathcal{H} (\hat m_1^{(k)}) (\omega^\ast) \right).
\]
We can now recover $\hat n_k$ from \eqref{m_new} which admits the from
\[
\hat n_{k}(\omega) = \hat n_{k-1}(\omega) \frac{\tilde \alpha_0(\omega) e^{i\frac{\omega}{c}\left( \hat n_{k-1}(\omega)2 z_k -\hat n_0 z_d \right)} -\tilde{\hat{m}}^{(k)}(\omega) }{\tilde \alpha_0(\omega) e^{i\frac{\omega}{c}\left( \hat n_{k-1}(\omega)2 z_k -\hat n_0 z_d \right)} +\tilde{\hat{m}}^{(k)}(\omega)},
\]
by equating the real and imaginary parts. 
\item[Step 2:] We update the data as in the non-absorbing case.

\end{description}

\section{Numerical Implementation}\label{sec4}

We solve the direct problem considering two different schemes depending on the properties of the refractive index, see \eqref{rei} and \eqref{eq_refract}. First, we consider the time-independent refractive index and phase-less data from a time-domain OCT system. 

\subsection{Reconstructions from phase-less data in time domain}

We model the incident field as a gaussian wave centered around a frequency $\omega_0$ moving in the $z$-direction of the form
\begin{equation}\label{field_inc}
u_0 (t,z) = e^{-\frac{(z-z_0 -ct)^2}{2\sigma^2}} \cos \left(\frac{\omega_0}{c} (z-z_0 - ct) \right),
\end{equation}
with width $\sigma,$ where $z_0$ denotes the source position and $c \approx 3\times 10^8$m/s is the speed of light.

The simulated data are created by solving \eqref{l53} using a finite difference scheme. We restrict $z \in [0,1.5]$mm and we set $T >0,$ the final time.  We consider absorbing boundary conditions at the end points and we set $u(0,z) = u_0 (0,z)$ and $\partial_t u (0,x) = 0$ as initial conditions. The left-going wave is ignored. We consider equidistant grid points  with step size $\Delta z = \lambda_0 /100,$ where $\lambda_0 = 2\pi c/\omega_0,$ is the central wavelength, and time step $\Delta t$ such that the CFL condition is satisfied. 

\begin{figure}[t]
\begin{center}
\includegraphics[width=0.9\textwidth]{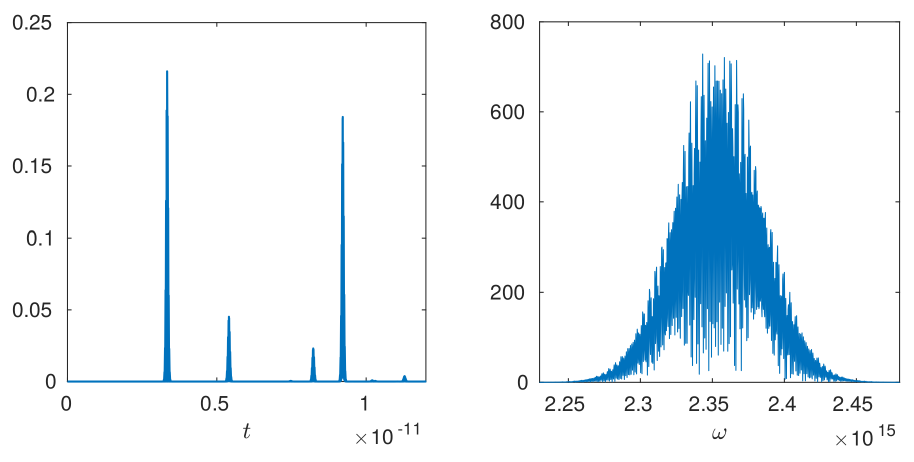}
\caption{The simulated data (absolute value) in the time-domain for the first example (left) and in the frequency-domain for the second example (right).}\label{fig1}
\end{center}
\end{figure}

The measurement data are given by
\begin{equation}\label{data_time}
m (t) = |(u-u_0)(t, z_d)|, \quad t \in (0,T],
\end{equation}
where $z_d$ denotes the position of the point detector. We add noise with respect to the $L^2$ norm
\[
m_\delta = m + \delta \frac{\norm{m}_2}{\norm{v}_2}v,
\]
where $v$ is a vector with components normally distributed random variables and $\delta$ denotes the noise level. We have to stress that the total time is such that the data contain also information from the multiple reflections inside the medium. We define the length of the $k-$th layer 
\[
\ell_k = z_{k+1} - z_k, \quad \mbox{for} \quad k = 1,...,N,
\]
and we set $\ell_{-1} = |z_d - z_0|$, and $\ell_0 = z_0 - z_d.$ We denote by
\[
\rho_k = \frac{n_{k-1}-n_k}{n_{k-1}+n_k}, \quad \mbox{for} \quad k = 1,...,N,
\]
the reflection coefficient at the interface $z=z_k.$

Here, we assume that we know only $n_0 = 1,$ and the positions of the source and the detector. Thus, we aim for recovering the position, the size and the optical properties of the medium. 

The proposed iterative scheme for a $N$-layer medium is presented in Algorithm \ref{alg1}, where the output is the reconstructed refractive indices and lengths of the layers.  First,  we order the observed $``$peaks$"$ at the image with respect to time, producing the set of data $(t_l, \, p_l), $ for $l=1,2,...,\Lambda.$ The number $\Lambda \geq N$ describes the number of single and multiple reflections arrived at the detector before the final time $T$. In order to obtain a physically compatible solution we impose some bounds $[\underline{n},\, \overline{n}]$ on the refractive index. This condition is not necessary for data with phase information. We update the data by neglecting the multiple reflections. To do so, once we have recovered the length and the refractive index of a layer, we neglect the $``$peaks$"$ appearing later referring to multiple reflections inside this layer. Of course, because of numerical error and noisy data, we give a tolerance depending on the time duration of the wave.

\renewcommand{\arraystretch}{1.6}
\setlength{\tabcolsep}{0.5em}
\begin{table}[t]
\begin{center}
\begin{tabular}{|c|c|c|c|c|} \hline
length (mm) & $\ell_0$ &    $\ell_1$ & $\ell_2$ &    $\ell_3$  \\ \hline
 exact        &     0.50000  &   0.20000 &    0.30000   &   0.10000  \\ \hline 
reconstructed  (noise free)    &         0.50000   &  0.20004 &    0.29990 &     0.10006    \\ \hline
reconstructed ($5\%$ noise)    &         0.49960  &   0.20021 &    0.30030  &    0.09976     \\ \hline \hline
refractive index & $n_1$ & $n_2$ & $n_3$ & $n_4$ \\ \hline
  exact        &     1.55000  &   1.41000  &  1.48000  & 1.00000    \\ \hline 
reconstructed  (noise free)   &       1.55107 &   1.41070 &  1.48087  & 1.00014    \\ \hline
reconstructed ($5\%$ noise)     &  1.55272  &  1.40740  & 1.48170  & 0.99851   \\ \hline
\end{tabular}
\end{center}
\caption{Reconstructed values using Algorithm \ref{alg1} for a three-layer medium.  }\label{table1}
\end{table}

We define the error function
\[
\epsilon = \left(\sum_{k=1}^N (n_k \ell_k - \tilde n_k \tilde \ell_k )^2\right)^{1/2},
\]
where $(n_k,\, \ell_k)$ and $(\tilde n_k,\, \tilde \ell_k),$ for $k=1,...,N$ are the exact and the reconstructed values, respectively. 

In the first example, we consider a three-layer medium positioned at $\ell_0 = 0.5$mm, with $\ell_{-1} =0.$ We set $(n_1, \, n_2, \, n_3) = (1.55, \, 1.41, \, 1.48)$ and lengths $(\ell_1, \, \ell_2, \, \ell_3) = (0.2, \, 0.3, \, 0.1)$mm. The obtained data \eqref{data_time} for this example  are given at the left picture in Fig. \ref{fig1}.

The results are presented in Table \ref{table1} for $[\underline{n},\, \overline{n}] = [1.345, 2]$ and $\mbox{tol} = 0.1$ps. We obtain accurate and stable reconstructions, with $\epsilon = 3.34\times 10^{-7}.$ This algorithm can be easily applied to multi-layer media and it is presented here since it will be the core of the more complicated algorithm in the Fourier domain.

\begin{algorithm}[t]
 \caption{Iterative scheme (in time) for phase-less data.}\label{alg1}
\KwResult{ $\tilde n_k$ and $\tilde \ell_{k-1},$ for $k=1,...,N. $ }
\textbf{Input:} $k=0, \, \rho_0 =0, \, n_0 =1,\, \ell_{-1}, \, \mbox{tol}$ and $(t_l, \, p_l),$ for $l=1,...,\Lambda$; \\
 \While{$k \leq N$}{
 \tcc{Step 1: Reconstruction of the refractive index.}
$ \rho_{k+1} = \frac{p_{k+1}}{\prod_{j=1}^k (1- \rho_j^2)}, \quad  \tilde{n}_{k+1}  = \tilde{n}_k  \frac{1- \rho_{k+1}}{1 +\rho_{k+1}};$ \\
  \If{$\tilde{n}_{k+1} \not\in [\underline{n},\, \overline{n}]$}{
$ \rho_{k+1} = -\rho_{k+1}, \quad \tilde{n}_{k+1}  = \tilde{n}_k  \frac{1- \rho_{k+1}}{1 +\rho_{k+1}}; $ \\
  }
  \tcc{Step 2: Reconstruction of the length.}
$ \tilde \ell_k =  \frac1{2}\left(c\left(\frac{t_{k+2}-t_{k+1}}{n_k}\right) - \tilde \ell_{k-1}\right); $ \\
     \tcc{Step 3: Update the data.}
   \For{$j = 1: \lfloor\Lambda/N\rfloor$}{
   $ \tau_{k+1} = t_{k+1} + j \frac{2 \tilde{n}_k \tilde\ell_k}{c}; $ \\
  \For{ $\kappa = k:\Lambda$}{
\If{$|t_\kappa - \tau_{k+1}| < \mbox{tol}$}{ 
  $ p_\kappa = 0; $ \\  
}  
}
}
k = k +1\;
}
\end{algorithm}

\subsection{Reconstructions having phase information in frequency domain}

We aim to reconstruct the time-dependent refractive index, meaning its frequency-dependent Fourier transform. As discussed already in Sec. \ref{subsec_phase}, it is possible from the phase-less OCT data to recover the full information, implying that we consider as measurement data the function
\begin{equation}\label{eq_freq_data}
\hat m (\omega) = (\hat u- \hat u_0) (\omega, z_d), \quad \omega \in [\underline{\omega},\overline{\omega}].
\end{equation}
Here, the frequency interval $[\underline{\omega},\overline{\omega}],$ with $\overline{\omega} >\underline{\omega} \gg 0,$ models the OCT data, recorded by a CCD camera placed after a spectrometer with wavelength range $[2\pi c/\overline{\omega}, 2\pi c/\underline{\omega} ],$ in a frequency-domain OCT system.

\subsubsection{Non-dispersive medium}

In order to construct the data \eqref{eq_freq_data}, we consider the time-dependent back-reflected field derived in the previous section, we add noise and we take its Fourier transform with respect to time. Then, we truncate the signal at the interval $[\underline{\omega},\overline{\omega}],$ see the right picture in Fig. \ref{fig1}.

\begin{figure}[t]
\begin{center}
\includegraphics[width=0.9\textwidth]{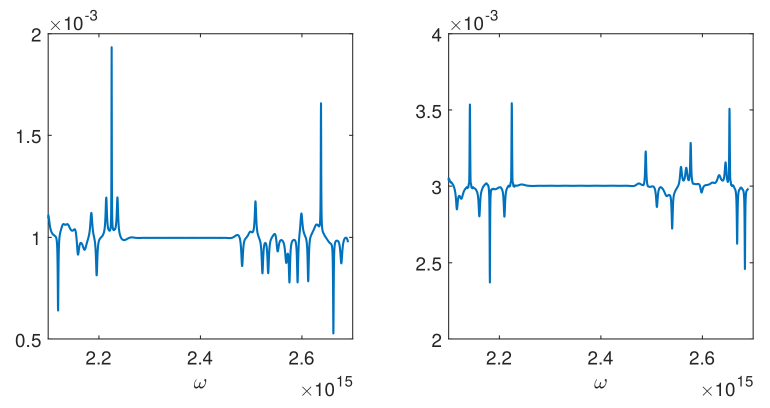}
\caption{The function $\phi(\omega),$ for $\omega \in [\underline{\omega},\overline{\omega}],$ at the first (left) and the last (right) iteration step of Algorithm  \ref{alg2}.}\label{fig2}
 \vspace{2cm}
\includegraphics[width=0.9\textwidth]{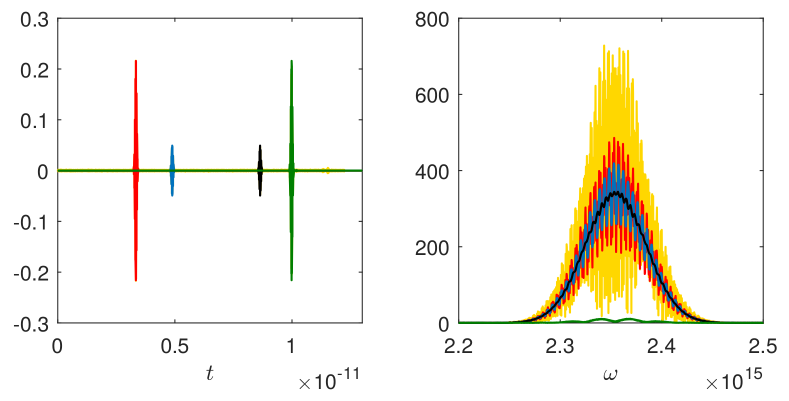}
\caption{The step 3 of the Algorithm \ref{alg2}, where we update the data. The (color) curve (right) is the signal in the frequency domain if we neglect the (color) and all the previous $``$peaks$"$ in the time-domain signal (left).}\label{fig3}
\end{center}
\end{figure}

In Algorithm \ref{alg2} we present the main steps of the iterative scheme as described in Sec. \ref{subsec_dis}. In Step 1, we take advantage of the causality property of the time-dependent signal and we zero-pad $\hat m(\omega),$ for all $\omega \in \R \setminus [\underline{\omega},\overline{\omega}],$ and then we recover the signal as two times the real part of the inverse Fourier transformed field. In the second step, we initially approximated the derivative with respect to frequency using finite differences but it did not produce nice reconstructions due to the highly oscillating signal. We replace the derivative with a high-order differentiator filter taking into account the sampling rate of the signal. In Fig. \ref{fig2}, we plot the function $\phi(\omega), \, \omega \in [\underline{\omega},\overline{\omega}],$ and we see that it is constant in a central interval (called trusted) and oscillates close to the end points. Thus, we denote by $\omega^\ast$ either a chosen frequency in the trusted interval or the mean of the frequencies in this trusted interval. We address here that one could also average over the whole spectrum and still get reasonable results.

\begin{algorithm}[t]
\KwResult{ $\tilde n_k$ and $\tilde \ell_{k-1},$ for $k=1,...,N. $ } 
\textbf{Input: }$k=0, \, \rho_0 =0, \, n_0 =1,\, \ell_{-1}, \, \mbox{tol},$ and $\hat m(\omega) ,$ for $\omega \in [\underline{\omega},\overline{\omega}]$;  \\
 \While{$k \leq N$}{
 \tcc{Step 1: Reconstruction of the refractive index.}
zero-padding and IFFT of the signal $\hat m$\;
Isolate the first peak and FFT the signal to obtain $\hat m^{(k)}$\;
$\rho_{k+1} = \frac{\max |\hat m^{(k)}|}{\max (\alpha)}, \quad  \tilde{n}_{k+1}  = \tilde{n}_k  \frac{1- \rho_{k+1}}{1 +\rho_{k+1}};$\\
  \If{$\tilde{n}_{k+1} \not\in [\underline{n},\, \overline{n}]$}{
$ \rho_{k+1} = -\rho_{k+1}, \quad \tilde{n}_{k+1}  = \tilde{n}_k  \frac{1- \rho_{k+1}}{1 +\rho_{k+1}}; $ \\
  }
 \tcc{Step 2: Reconstruction of the length.}
 define $f_k (\omega) = \frac{\hat m^{(k)}}{\alpha} / |\frac{\hat m^{(k)}}{\alpha} |;$\\
 \If{$k=0$}{
 $d_k = n_0 \ell_{-1};$\\
\Else
    {
    $d_k = n_0 \ell_{-1} - 2 \tilde n_{k-1} \tilde \ell_{k-1};$
    } 
 }
  $\phi(\omega) = c |\partial_\omega f_k (\omega)|, \quad \tilde \ell_k = -\frac{-\phi(\omega^\ast) - d_k}{2\tilde{n}_k} ;$\\
  \If{$\tilde \ell_k<0$}{
$\tilde \ell_k = -\frac{\phi(\omega^\ast) - d_k}{2\tilde{n}_k};$\\  
  }
     \tcc{Step 3: Update the data.}
 $\hat m (\omega) = (\hat m (\omega)- \hat m^{(k)}(\omega))/ ( 1- \rho_{k+1}^2);$\\
k = k +1\;
}
 \caption{Iterative scheme (in frequency) using phase information for non-dispersive medium.}\label{alg2}
\end{algorithm}

In the second example, we consider again a three-layer medium with parameters  $(n_1, \, n_2, \, n_3) = (1.55, \, 1.405, \,$ $1.55)$ and lengths $(\ell_1, \, \ell_2, \, \ell_3) = (0.15, \, 0.5,$  $0.13)$mm. Here,  $\ell_0 = 0.7$mm and  $\ell_{-1} =0.2$mm. The central frequency is given by $\omega_0 = 2\pi c/ \lambda_0,$ with $\lambda_0 = 800$nm. The sampling rate is $f_s = 100 c/\lambda_0.$ The recovered parameters for noise-free and noisy data are presented in Table \ref{table2}. The relative error is $\epsilon = 2.164\times 10^{-5}.$ In Fig. \ref{fig3}, we see how the data change as the Algorithm \ref{alg2} progresses. The picture on the right shows the data in frequency domain with respect to the $``$peaks$"$ presented in the left picture where we see the time-domain data. The yellow curve (right) represents the full data, the red curve (right) the data if we neglect the red $``$peak$"$ (left), the blue curve shows the data if we neglect also the blue $``$peak$"$, and so on. The green curve (right) represents the signal from the multiple reflections. We observe that the OCT signal maintains the Gaussian form of the incident wave, centered around the central frequency, and the different reflections result to the oscillations of the field.

 \renewcommand{\arraystretch}{1.6}
\setlength{\tabcolsep}{0.5em}
\begin{table}[t]
\begin{center}
\begin{tabular}{|c|c|c|c|c|} \hline
length (mm) & $\ell_0$ &    $\ell_1$ & $\ell_2$ &    $\ell_3$  \\ \hline
 exact        &     0.70000  &   0.15000 &    0.40000   &   0.13000  \\ \hline 
reconstructed  (noise free)    &         0.69885   &  0.15210 &    0.39628 &     0.13451    \\ \hline
reconstructed ($5\%$ noise)    &         0.70340 &   0.15387 &    0.39337  &    0.13669    \\ \hline \hline
refractive index & $n_1$ & $n_2$ & $n_3$ & $n_4$ \\ \hline
  exact        &     1.55000  &   1.40500  &  1.55000  & 1.00000    \\ \hline 
reconstructed  (noise free)   &       1.55107 &   1.40568 &  1.55107  & 0.99782     \\ \hline
reconstructed ($5\%$ noise)     &  1.55164  &  1.40599  &  1.55139  &  0.99695       \\ \hline
\end{tabular}
\end{center}
\caption{Reconstructed values using Algorithm \ref{alg2} for a three-layer medium. }\label{table2}
\end{table}

\subsubsection{Dispersive medium}

The incident field in the frequency domain takes the form
\begin{equation*}
\hat u_0 (\omega,z) =  \sqrt{2\pi} \frac{\sigma}{2c} e^{-\frac{\sigma^2 (\omega - \omega_0)^2}{2c^2}}e^{i \tfrac{\omega}{c} (z-z_0)}, \quad \omega >0, \, z\in \R, 
\end{equation*}
which is the Fourier transform with respect to time of $u_0,$ given by \eqref{field_inc}, restricted to positive frequencies. This field describes a plane wave moving in the $z-$direction having a Gaussian profile perpendicular to the incident direction, centered around $\omega_0.$ We generate the data considering the formula \eqref{l43} and then we add noise. The Algorithm \ref{alg3} summarizes the steps of the iterative scheme, which for a frequency-independent refractive index simplifies to Algorithm \ref{alg2}.

We model the wavelength-dependent refractive index of the medium using the standard formula, known as Cauchy's equation, 
\[
n (\lambda)   = \beta_1 + \frac{\beta_2}{\lambda^2} + \frac{\beta_3}{\lambda^4},
\]
for some fitting coefficients $\beta_j , \, j=1,2,3.$ In Fig. \ref{fig_refra}, we see the exact refractive index of the first (left) and the third (right) layer for the medium used in the third example. Afterwards, we consider the refractive index as a function of frequency. 

\begin{figure}[t]
\begin{center}
\includegraphics[width=0.9\textwidth]{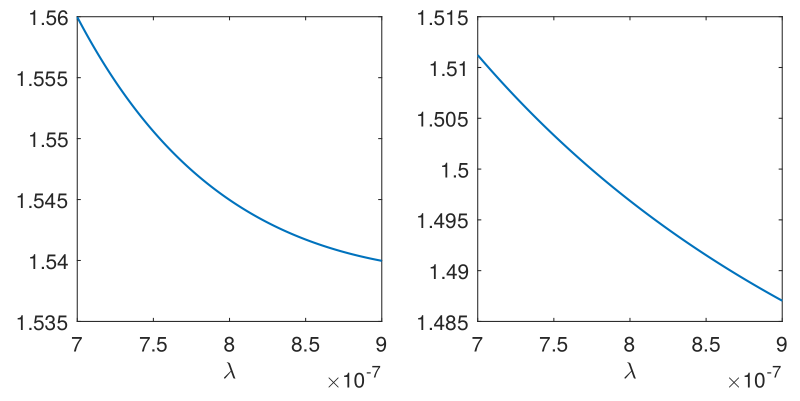}
\caption{The behavior of the medium with respect to wavelength (dispersion). The refractive index of the first (left) and the third (right) layer for the third example in the range $[700,\,900]$nm of the spectrometer.
}\label{fig_refra}
 \vspace{2cm}
\includegraphics[width=0.9\textwidth]{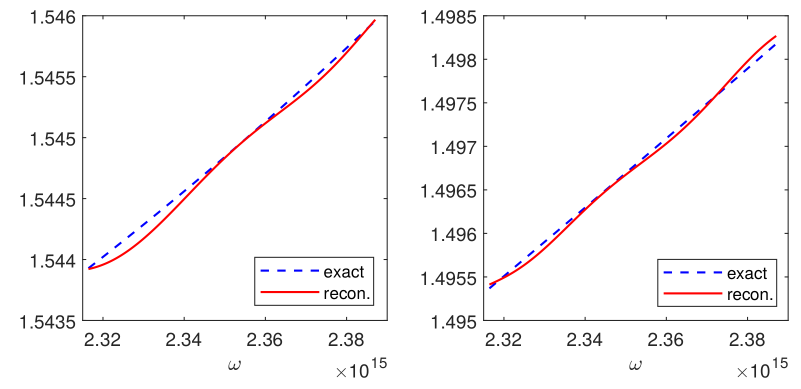}
\caption{The exact (dashed blue line) and the reconstructed (red solid line) refractive index of the first (left) and the third (right) layer. These are the results of the Algorithm \ref{alg3} for noisy data.
}\label{fig_results}
\end{center}
\end{figure}

As already discussed, the calculations close to the end points were not stable and since here we are interested  in reconstructing the frequency-dependent refractive index, we restrict the computational domain and then we extrapolate the recovered functions in order to update the data. 

The medium lengths are given by $(\ell_1, \, \ell_2, \, \ell_3) = (0.2,\, 0.3, \, 0.1)$mm, and we set $\ell_0 = 0.7$mm and  $\ell_{-1} =0.$ The second layer has constant refractive index given by $n_2 (\omega) =1.41.$ The reconstructions of $n_1 (\omega)$  and $n_3 (\omega)$ are presented in Fig. \ref{fig_results} for data with $2\%$ noise. In Table \ref{table3}, we see the recovered lengths and the refractive indices at specific frequencies.

\begin{algorithm}[H]
\KwResult{ $\tilde n_k (\omega)$ and $\tilde \ell_{k-1},$ for $k=1,...,N. $ }
\textbf{Input:} $k=0, \, \rho_0 =0, \, n_0 =1,\, \ell_{-1}, \, \mbox{tol}, \, \mathcal{W} \subset [\underline{\omega},\overline{\omega}]$ and $\hat m(\omega) ,$ for $\omega \in [\underline{\omega},\overline{\omega}]$; \\
 \While{$k \leq N$}{
 \tcc{Step 1: Reconstruction of the refractive index.}
zero-padding and IFFT of the signal $\hat m$\;
Isolate the first peak and FFT the signal to obtain $\hat m^{(k)}$\;
$\rho_{k+1} (\omega)= \frac{ |\hat m^{(k)} (\omega)|}{ \alpha  (\omega)}, \quad  \tilde{n}_{k+1}  (\omega)= \tilde{n}_k  (\omega)  \frac{1- \rho_{k+1} (\omega)}{1 +\rho_{k+1} (\omega)}, \quad \omega\in\mathcal{W};$\\
  \If{$\mbox{max}\{\tilde{n}_{k+1}  (\omega)\} > \overline{n}$ \textbf{or} $\mbox{min}\{\tilde{n}_{k+1}  (\omega)\} < \underline{n}$}{
$ \rho_{k+1}  (\omega)= -\rho_{k+1}  (\omega), \quad \tilde{n}_{k+1} (\omega)  = \tilde{n}_k  (\omega)  \frac{1- \rho_{k+1} (\omega)}{1 +\rho_{k+1} (\omega)}; $ \\
  }
 \tcc{Step 2: Reconstruction of the length.}
 define $f_k (\omega) = \frac{\hat m^{(k)}}{\alpha} / |\frac{\hat m^{(k)}}{\alpha} |, \quad \omega\in\mathcal{W};$\\
 \If{$k=0$}{
 $d_k (\omega) = n_0 \ell_{-1};$\\
\Else
    {
    $d_k  (\omega)= n_0 \ell_{-1} - 2 (\tilde n_{k-1}  (\omega)+\omega \partial_\omega \tilde n_{k-1} (\omega))\tilde \ell_{k-1};$
    } 
 }
  $\phi(\omega) = c |\partial_\omega f_k (\omega)|, \quad \psi_k (\omega) = -\frac{-\phi(\omega^\ast) - d_k (\omega)}{2(\tilde{n}_k+\omega \partial_\omega \tilde n_{k})} ;$\\
  \If{$\mbox{max}\{\psi_k \}<0$}{
$\psi_k (\omega)= -\frac{\phi(\omega^\ast) - d_k (\omega)}{2(\tilde{n}_k+\omega \partial_\omega \tilde n_{k})};$\\  
  }
  $\tilde{\ell}_k = \psi_k (\omega^\ast);$\\
     \tcc{Step 3: Update the data.}
     extrapolate $\rho_{k+1}(\omega)$ from $\mathcal{W}$ to $[\underline{\omega},\overline{\omega}];$ \\
 $\hat m (\omega) = (\hat m (\omega)- \hat m^{(k)}(\omega))/ ( 1- \rho_{k+1}^2 (\omega));$\\
k = k +1\;
}
 \caption{Iterative scheme (in frequency) using phase information for dispersive medium.}\label{alg3}
\end{algorithm}

 \renewcommand{\arraystretch}{1.6}
\setlength{\tabcolsep}{0.5em}
\begin{table}
\begin{center}
\begin{tabular}{|c|c|c|c|c|} \hline
length (mm) & $\ell_0$ &    $\ell_1$ & $\ell_2$ &    $\ell_3$  \\ \hline
 exact        &     0.70000  &   0.20000 &    0.30000   &   0.10000  \\ \hline 
reconstructed ($2\%$ noise)    &        0.69790 &    0.20139 &   0.29569  &  0.10329      \\ \hline \hline
refractive index & $n_1 (\omega^\ast)$ & $n_2(\omega^\ast)$ & $n_3(\omega^\ast)$ & $n_4$ \\ \hline
  exact        &    1.54488 &   1.41000 &   1.49675 &   1.00000    \\ \hline 
reconstructed ($2\%$ noise)     &  1.54499 &   1.41094 &   1.49884 &   1.00484    \\ \hline
\end{tabular}
\end{center}
\caption{Reconstructed values using Algorithm \ref{alg3} for a three-layer medium. The values of the refractive indices are given at $\omega^\ast =2.351\times 10^{15}.$}\label{table3}
\end{table}

\section{Conclusions}

In this work we addressed the inverse problem of recovering the optical properties of a multi-layer medium from simulated data modelling a frequency-domain OCT system. We considered the cases of non-dispersive, dispersive and absorbing media. We proposed reconstruction methods and we presented numerical examples justifying the feasibility of the derived schemes. Stable reconstruction with respect to noise were presented. The methods are based on standard equations, equivalent to the Fresnel equations, and to ideas from stripping algorithms. The originality of this work lies in the combination of them into a new iterative method that addresses also the frequency-dependent case, which needs special treatment. 
As a future work, we plan to examine the  applicability of the iterative schemes for experimental data and test numerically the method for absorbing media.

\section*{acknowledgement}
The work of PE and LV was supported by the Austrian Science Fund (FWF) in the project F6804--N36 (Quantitative Coupled Physics Imaging) within the Special Research Programme SFB F68: $``$Tomography Across the Scales$".$
%

\section*{References}
\printbibliography[heading=none]

\end{document}